\documentclass{article}%
\usepackage{amssymb}
\usepackage{amsfonts}
\usepackage{amsmath}
\usepackage{graphicx}%
\providecommand{\U}[1]{\protect\rule{.1in}{.1in}}
\newtheorem{theorem}{Theorem}

\newtheorem{corollary}[theorem]{Corollary}

\newtheorem{definition}[theorem]{Definition}
\newtheorem{example}[theorem]{Example}

\newtheorem{lemma}[theorem]{Lemma}

\newtheorem{proposition}[theorem]{Proposition}
\newtheorem{remark}[theorem]{Remark}

\newenvironment{proof}[1][Proof]{\noindent\textbf{#1.} }{\ \rule{0.5em}{0.5em}}
\begin{document}

\title{ Weak forms of topological and measure theoretical equicontinuity:
relationships with discrete spectrum and sequence entropy}
\date{\vspace{-5ex}}
\author{Felipe Garc\'{\i}a-Ramos\\felipegra@math.ubc.com\\University of British Columbia}
\maketitle

\begin{abstract}
We define weaker forms of topological and measure theoretical equicontinuity
for topological dynamical systems, and we study their relationships with
sequence entropy and systems with discrete spectrum.

We show that for topological systems equipped with ergodic measures having
discrete spectrum is equivalent to $\mu$-mean equicontinuity.

In the purely topological category we show that minimal subshifts with zero
topological sequence entropy are strictly contained in the diam-mean
equicontinuous systems; and that transitive almost automorphic subshifts are
diam-mean equicontinuous if and only if they are regular (i.e. the maximal
equicontinuous factor map is 1-1 on a set of full Haar measure).

For both categories we find characterizations using stronger versions of the
classical notion of sensitivity. As a consequence we obtain a dichotomy
between discrete spectrum and a strong form of measure theoretic sensitivity.

\end{abstract}
\tableofcontents

\section{Introduction}

A topological dynamical system (TDS), $(X,T)$, is a continuous action $T$ on a
compact metric space $X$. The dynamical behaviour of these systems can range
from very rigid to very chaotic. Equicontinuity represents predictability. A
TDS is equicontinuous if the family $\{T^{i}\}$ is equicontinuous, or
equivalently, if whenever two points $x,y\in X$ are close, then $T^{i}%
(x),T^{i}(y)$ stay close for \emph{all} $i$. The prototype for an
equicontinuous TDS is a rotation on a compact abelian group, and it is well
known that any transitive equicontinuous TDS is topologically conjugate to
such a rotation. Sensitive dependence on initial conditions (sensitivity) is
considered a weak form of chaos. Auslander-Yorke \cite{auslanderyorke} showed
that a minimal TDS is either sensitive or equicontinuous. In \cite{Fominst}
Fomin introduced a weaker form of equicontinuity called mean-L-stable (or mean
equicontinuity) which requires that if $x,y\in X$ are close then
$T^{i}(x),T^{i}(y)$ stay close for \emph{many} $i$.

A classical result of Halmos and von Neumann~\cite{halmos1942operator} states
that an ergodic measure-preserving transformation (MPT) $T$ has discrete
spectrum if and only if it is measure-theoretically isomorphic to a rotation
$S$ on a compact abelian group; here, the measure on the group is the Haar
probability measure, and the spectrum refers to the spectrum of the operator
induced by $T$ on $L^{2}$.

Consider the hybrid situation of a TDS that is also an MPT, i.e., a continuous
map $T$ on a compact metric space $X$ endowed with a Borel probability measure
$\mu$ such that $T$ preserves $\mu$. Physical models of systems at very low
temperatures, like quasicrystals, can be modelled by TDS with discrete
spectrum \cite{quasicrystals}.\ If an ergodic TDS $T$ has discrete spectrum,
it is natural to ask how much of the equicontinuity of a rotation, as a TDS,
must be preserved by the isomorphism between $T$ and the rotation.

Gilman~\cite{Gilman1}\cite{Gilman2} introduced a notion of $\mu$%
-equicontinuity for cellular automata and later
Huang-Lu-Ye~\cite{mtequicontinuity} introduced a different definition of $\mu
$-equicontinuity (which under some conditions are equivalent,
see~\cite{mueqtds}) and showed that $\mu$-equicontinuous systems have discrete
spectrum. We introduce a weakening of both $\mu$-equicontinuity and mean
equicontinuity that we call $\mu$-mean equicontinuity, and we show that if
$(X,\mu,T)$ is an ergodic system then it has discrete spectrum if and only if
it is $\mu$-mean equicontinuous (Corollary \ref{principal}).

For this result, we make use of Kushnirenko's characterization of MPT's with
discrete spectrum as those with zero measure-theoretic sequence
entropy~\cite{0036-0279-22-5-R02}.

We also define $\mu-$mean sensitivity, and we show that ergodic topological
systems are either $\mu-$mean equicontinuous or $\mu-$mean sensitive. This
implies that every ergodic TDS $(X,\mu,T)$ is either $\mu-$mean sensitive or
has discrete spectrum (Corollary \ref{principal}). These results can be
interpreted in two different ways. First that systems with discrete spectrum
are predictable in the sense that they are $\mu-$mean equicontinuous; and that
systems that don't have pure discrete spectrum are chaotic in the sense that
the are $\mu-$mean sensitive. As a corollary we can develop a notion of
sensitivity for purely measure preserving transformations. We show that an
ergodic MPT is either measurably sensitive or has discrete spectrum (Theorem
\ref{measuret}).

We may ask if some of these results hold at the topological level. The
topological version of Halmos-Von Neumann Theorem states that for transitive
TDS, equicontinuous maps can be characterized as those with topological
discrete spectrum, i.e., the induced operator on $C(X)$ has discrete
spectrum~(see e.g. \cite{walters2000introduction}). It is easy to see that any
equicontinuous TDS has zero topological entropy. Similar to the
measure-theoretic sequence entropy one can define topological sequence
entropy. A null system is a TDS that has zero topological sequence entropy. It
is well known that equicontinuity implies nullness, but the converse is
false~\cite{goodman1974topological}. Nevertheless, one can ask if there is a
sense in which every null TDS is \textquotedblleft nearly\textquotedblright%
\ equicontinuous. Indeed, in the minimal case there is. Any TDS has a unique
maximal equicontinuous factor \cite{auslander}, and Huang-Li--Shao-Ye
\cite{huang2003null} showed that for any minimal null TDS $(X,T)$, the factor
map from $X$ to its maximal equicontinuous factor is 1-1 on a residual set
(i.e., $(X,T)$ is almost automorphic). We strengthen this result for subshifts
in Corollary \ref{nullquasi}, by showing that the factor map is 1-1 on a set
of full Haar measure (i.e., $(X,T)$ is regular).

In order to establish Corollary \ref{nullquasi}, we introduce another weak
topological form of equicontinuity that we call diam-mean equicontinuity
(stronger than mean equicontinuity). We show that for minimal TDS, nullness
implies a form of diam-mean equicontinuity (Corollary \ref{nullweak}) and that
an almost automorphic subshift is diam-mean equicontinuous if and only if it
is regular (Theorem \ref{aa}).

In conclusion for minimal subshifts we have the following implications:
\begin{align*}
\text{Top. discrete spectrum }  &  =\text{equicontinuity}\\
&  \Longrightarrow\text{nullness}\\
&  \Longrightarrow\text{diam-mean equicontinuity}\\
&  \Longrightarrow\text{mean equicontinuity}\\
&  \Longrightarrow\mu-\text{mean equicontinuity (for every ergodic measure
}\mu\text{)}\\
&  =\text{every ergodic measure has discrete spectrum}%
\end{align*}

In Section 5 we explain how these results can be generalized to amenable
semigroup actions.

Mean equicontinuous $\mathbb{Z-}$systems were recently studied in
\cite{li2013mean}. They independently obtain Theorem \ref{sensi} and they
proved that if $(X,T)$ is mean equicontinuous and transitive and $\mu$ is
ergodic then the system has discrete spectrum. This was an open question from
\cite{scarpellini1982stability}. This result can also be obtained with
Corollary \ref{principal}.

\textbf{Acknowledgements: }I would like to thank Karl Petersen and Tomasz
Downarowicz for conversations that inspired this paper, Brian Marcus for his
suggestions and always enjoyable conversations, and Boris Solomyak, Xiangdong
Ye and the referee for their valuable comments.

The author of this paper was supported by a CONACyT PhD fellowship.

Starting on August 2015 the author is now affiliated to IMPA and University of
Sao Paulo. 

\section{Topological dynamical systems}

A $\mathbb{G-}$\textbf{topological dynamical system (}$\mathbb{G-}%
$\textbf{TDS)} is a pair $(X,T),$ where $X$ is a compact metric space,
$\mathbb{G}$ a semigroup and $T:=\left\{  T^{i}:i\in\mathbb{G}\right\}  $ is a
$\mathbb{G}-$ continuous action on $X.$ If $\mathbb{G=Z}_{+}^{d}$, we simply
say $(X,T)$ is a TDS$.$ In Sections 2, 3 and 4 we use $\mathbb{G}$
$=\mathbb{Z}_{+}^{d}.$ All the results hold for countable discrete abelian
actions; some are more general, see Section 5 for details.

The metric and $\varepsilon-$closed balls on a compact metric space $X$ will
be denoted by $d$ and $B_{\varepsilon}(x)$ respectively.

\bigskip Mathematical definitions of chaos have been widely studied. Most of
them require the system to be sensitive. A TDS $(X,T)$ has \textbf{sensitive
dependence on initial conditions (}or is \textbf{sensitive)} if there exists
$\varepsilon>0$ such that for every open set $A\subset X$ there exists $x,y\in
A$ and $i\in$ $\mathbb{G}$ such that $d(T^{i}x,T^{i}y)>\varepsilon.$ \ On the
other hand equicontinuity represents predictable behaviour. A TDS is
\textbf{equicontinuous} if $T$ is an equicontinuous family. Auslander-Yorke
showed that a minimal TDS is either sensitive or equicontinuous
\cite{auslanderyorke}. A problem with this classification is that
equicontinuity is a strong property and not adequate for subshifts; a subshift
is equicontinuous if and only if it is finite (see for example \cite{mueqca}).

\subsection{Mean equicontinuity and mean sensitivity}

\bigskip In this section we define mean equicontinuity and mean sensitivity
and we adapt Auslander-Yorke's dichotomy to this set up.

\begin{definition}
Let $S\subset\mathbb{G}$. We denote with $F_{n}$ the $n-$cube $\left[
0,n\right]  ^{d}.$ We define the \textbf{lower density of S }as%
\[
\underline{D}(S):=\liminf_{n\rightarrow\infty}\frac{\left\vert S\cap
F_{n}\right\vert }{\left\vert F_{n}\right\vert },
\]
and the \textbf{upper density of S }as
\[
\overline{D}(S):=\limsup_{n\rightarrow\infty}\frac{\left\vert S\cap
F_{n}\right\vert }{\left\vert F_{n}\right\vert }.
\]

\end{definition}

The following properties are easy to prove and will be used throughout the paper.

\begin{lemma}
\label{basico}Let $S,S^{\prime}\subset\mathbb{G}$, $i\in\mathbb{G}$, and
$F\subset\mathbb{G}$ finite set$.$ We have that

$\cdot\underline{D}(S)=\underline{D}(i+S)$ and $\overline{D}(S)=\overline
{D}(i+S).$

$\cdot$\underline{$D$}$(S)+\overline{D}(S^{c})=1.$

$\cdot$If $\underline{D}(S)+\overline{D}(S^{\prime})>1$ then $S\cap S^{\prime
}\neq\emptyset.$

$\cdot\underline{D}(S):=\liminf_{n\rightarrow\infty}\frac{\left\vert S\cap
F_{n}\diagdown F\right\vert }{\left\vert F_{n}\diagdown F\right\vert }$

$\cdot\overline{D}(S):=\limsup_{n\rightarrow\infty}\frac{\left\vert S\cap
F_{n}\diagdown F\right\vert }{\left\vert F_{n}\diagdown F\right\vert }$
\end{lemma}

\begin{definition}
Let $(X,T)$ be a TDS. We say $x\in X$ is a \textbf{mean equicontinuous point}
if for every $\varepsilon>0$ there exists $\delta>0$ such that if $y\in
B_{\delta}(x)$ then
\[
\overline{D}(i\in\mathbb{G}:d(T^{i}x,T^{i}y)>\varepsilon)<\varepsilon
\]
(equivalently $\underline{D}(i\in\mathbb{G}:d(T^{i}x,T^{i}y)\leq
\varepsilon)\geq1-\varepsilon).$ We say $(X,T)$ is \textbf{mean equicontinuous
(or mean-L-stable)} if every $x\in X$ is a mean equicontinuous point. We say
$(X,T)$ is \textbf{almost mean equicontinuous }if the set of mean
equicontinuity points is residual.
\end{definition}

Mean equicontinuous systems were introduced by Fomin \cite{Fominst}. They have
been studied in \cite{auslander1959mean}, \cite{oxtoby1952},
\cite{scarpellini1982stability} and \cite{li2013mean}.

\bigskip Using the fact that a continuous function on a compact set is
uniformly continuous we will see that $(X,T)$ is mean equicontinuous if and
only if it is uniformly mean equicontinuous i.e. for every $\varepsilon>0$
there exists $\delta>0$ such that if $d(x,y)\leq\delta$ then $\overline
{D}(i\in\mathbb{G}:d(T^{i}x,T^{i}y)>\varepsilon)<\varepsilon$ (see Remark
\ref{uniform}).

\begin{definition}
We denote the set of mean equicontinuity points by $E^{m}$ and we define
\[
E_{\varepsilon}^{m}:=\left\{  x\in X:\exists\delta>0\text{ }\forall y,z\in
B_{\delta}(x),\text{ }\underline{D}\left\{  i\in\mathbb{G}:d(T^{i}%
y,T^{i}z)\leq\varepsilon\right\}  \geq1-\varepsilon\right\}  .
\]

\end{definition}

Note that $E^{m}=\cap_{\varepsilon>0}E_{\varepsilon}^{m}.$

\begin{lemma}
\label{invariant}Let $(X,T)$ be a TDS. The sets $E^{m}$, $E_{\varepsilon}^{m}$
are inversely invariant (i.e. $T^{-j}(E^{m})\subseteq E^{m}$, $T^{-j}%
(E_{\varepsilon}^{m})\subseteq E_{\varepsilon}^{m}$ for all $j\in\mathbb{G)}$
and $E_{\varepsilon}^{m}$ is open.
\end{lemma}

\begin{proof}
Let $j\in\mathbb{G},$ $\varepsilon>0,$ and $x\in T^{-j}E_{\varepsilon}^{m}.$
There exists $\eta>0$ such that if $d(T^{j}x,z)\leq\eta$ then $\underline{D}%
\left\{  i:d(T^{i+j}x,T^{i}z)\leq\varepsilon\right\}  \geq1-\varepsilon.$
There exists $\delta>0$ such that if $d(x,y)<\delta$ then $d(T^{j}%
x,T^{j}y)<\eta$ (and hence $\underline{D}\left\{  i:d(T^{i+j}x,T^{i+j}%
y)\leq\varepsilon\right\}  \geq1-\varepsilon).$ We conclude that $x\in
E_{\varepsilon}^{m}.$ This implies $E^{m}$ is also inversely
invariant.\newline

Let $x\in E_{\varepsilon}^{m}$ and $\delta>0$ be a constant that satisfies the
property of the definition of $E_{\varepsilon}^{m}.$ If $d(x,w)<\delta/2$ then
$w\in E_{\varepsilon}^{m}.$ Indeed if $y,z\in B_{\delta/2}(w)$ then $y,z\in
B_{\delta}(x).$
\end{proof}

\begin{definition}
A TDS $(X,T)$ is \textbf{mean sensitive} if there exists $\varepsilon>0$ such
that for every open set $U$ there exist $x,y\in U$ such that%
\[
\overline{D}(i\in\mathbb{G}:d(T^{i}x,T^{i}y)>\varepsilon)>\varepsilon.
\]

\end{definition}

\begin{definition}
Let $(X,T)$ be a TDS. We say $(X,T)$ is \textbf{transitive} if for every open
sets $U$ and $V$ there exists $i\in\mathbb{G}$ such that $T^{i}U\cap
V\neq\emptyset.$

We say $x\in X$ is a \textbf{transitive point} if $\left\{  T^{i}%
x:i\in\mathbb{G}\right\}  $ is dense. If every $x\in X$ is transitive then we
say the system is \textbf{minimal}.
\end{definition}

If $(X,T)$ is transitive then $X$ contains a residual set of transitive
points. If $X$ has no isolated points and $(X,T)$ has a transitive point then
$(X,T)$ is transitive \cite{kolyada334some}. If $(X,T)$ is sensitive then $X$
has no isolated points.

It is not hard to see that mean sensitive systems have no mean equicontinuity
points, as a matter of fact we have the following dichotomies.

\begin{theorem}
\label{sensi}A transitive system is either almost mean equicontinuous or mean
sensitive. A minimal system is either mean equicontinuous or mean sensitive.
\end{theorem}

\begin{proof}
Let $(X,T)$ be a transitive TDS.

If $(X,T)$ is not sensitive then by \cite{AkinAuslander} it is almost
equicontinuous and hence almost mean equicontinuous.

Let $(X,T)$ be a sensitive TDS (hence $X$ has no isolated points). We will
show that $E_{\varepsilon}^{m}$ is either empty or dense.

Assume $E_{\varepsilon}^{m}$ is non-empty and not dense. Then $U=X\diagdown
\overline{E_{\varepsilon}^{m}}$ is a non-empty open set. Since the system is
transitive and $E_{\varepsilon}^{m}$ is open (Lemma \ref{invariant}) there
exists $t\in\mathbb{G}$ such that $U\cap T^{-t}(E_{\varepsilon}^{m}$ $)$ is
non empty. By Lemma \ref{invariant} we have that $U\cap T^{-t}(E_{\varepsilon
}^{m}$ $)\subset U\cap E_{\varepsilon}^{m}=\emptyset;$ a contradiction.
\newline If $E_{\varepsilon}^{m}$ is non-empty for every $\varepsilon>0$ then
we have that $E^{m}=\cap_{n\geq1}E_{1/n}^{m}$ is a residual set; hence the
system is almost mean equicontinuous.

If there exists $\varepsilon>0$ such that $E_{\varepsilon}^{m}$ is empty, then
for any open ball $U=B_{\delta}(x)$ there exist $y,z\in B_{\delta}(x)$ such
that $\underline{D}\left\{  i\in\mathbb{G}:d(T^{i}y,T^{i}z)\leq\varepsilon
\right\}  \leq1-\varepsilon;$ this means that $\overline{D}\left\{
i\in\mathbb{G}:d(T^{i}y,T^{i}z)>\varepsilon\right\}  >\varepsilon.$ It follows
that $(X,T)$ is mean sensitive.

Now suppose $(X,T)$ is minimal and almost mean equicontinuous. For every $x\in
X$ and every $\varepsilon>0$ there exists $t\in\mathbb{G}$ such that
$T^{t}x\in E_{\varepsilon}^{m}.$ Since $E_{\varepsilon}^{m}$ is inversely
invariant $x\in E_{\varepsilon}^{m}$ and hence $x\in E^{m}.$
\end{proof}

An analogous result appeared in \cite{li2013mean} for $\mathbb{G=Z}_{+}$.

It will be useful to describe mean equicontinuity in terms of the Besicovitch pseudometric.

\begin{definition}
\label{besi}We define $\Delta_{\delta}(x,y):=\left\{  i\in\mathbb{G}:\text{
}d\mathbb{(}T^{i}x,T^{i}y)>\delta\right\}  $ and the \textbf{Besicovitch
pseudometric} as $d_{b}(x,y):=\inf\left\{  \delta>0:\overline{D}%
(\Delta_{\delta}(x,y))<\delta\right\}  .$ By identifying points that are at
pseudo-distance zero we obtain a metric space $\left(  X/d_{b},d_{b}\right)  $
that will be called the \textbf{Besicovitch space. }The projection
$f_{b}:(X,d)\rightarrow(X/d_{b},d_{b})$ will be called the \textbf{Besicovitch
projection. }The $\varepsilon-$closed balls of the Besicovitch pseudometric
will be denoted by $B_{\varepsilon}^{b}(x).$
\end{definition}

One can check that in fact this is a pseudometric using that $\overline
{D}(S)+\overline{D}(S^{\prime})\geq\overline{D}(S\cup S^{\prime})$.

It is not difficult to see that if $x\in X$ is a mean equicontinuous point
then $f_{b}$ is continuous at $x.$ This implies the Besicovitch projection is
continuous if and only if $(X,T)$ is mean equicontinuous.

\begin{remark}
\label{uniform}If $(X,T)$ is mean equicontinuous then $f_{b}$ is continuous
and hence $f_{b}$ is uniformly continuous; this means that $(X,T)$ is
uniformly mean equicontinuous i.e. for every $\varepsilon>0$ there exists
$\delta>0$ such that if $d(x,y)\leq\delta$ then \underline{$D$}$(i\in
\mathbb{G}:d(T^{i}x,T^{i}y)\leq\varepsilon)\geq1-\varepsilon.$
\end{remark}

\begin{remark}
The Besicovitch pseudometric is sometimes expressed with an equivalent metric
using averages. For example if $\mathbb{G=Z}_{+}$ then%
\[
d_{b}(x,y):=\limsup_{n\rightarrow\infty}\frac{1}{n}\sum_{i=1}^{n}%
d(T^{i}x,T^{i}y).
\]

\end{remark}

In \cite{blanchard1998cellular} equicontinuity with respect to the Besicovitch
pseudometric (of the shift) was studied for cellular automata; this is a
different property than mean equicontinuity.

It is well known that transitive equicontinuous systems are minimal. We give a
similar result by weakening one hypothesis and strengthening the other.

\begin{definition}
A TDS $(X,T)$ is \textbf{strongly transitive} if for every open set $U$ there
exists a transitive point $x\in U$ that returns to $U$ with positive lower density.
\end{definition}

\begin{theorem}
\label{sensitivity}Every strongly transitive mean equicontinuous system is minimal.
\end{theorem}

\begin{proof}
Let $x,y\in X$ and $\varepsilon>0.$ Since the system is strongly transitive
there exists a transitive point $z\in B_{\varepsilon/2}(y)$ such that
$a:=\underline{D}\left\{  i:T^{i}z\in B_{\varepsilon/2}(y)\right\}  >0.$ Since
the system is mean equicontinuous there exists $\delta>0$ such that if $w\in
B_{\delta}(x)$ then $d_{b}(x,w)\leq\min\left\{  \varepsilon/2,a\right\}  .$
There exists $t_{1}\in\mathbb{G}$ such that $T^{t_{1}}z\in B_{\delta}(x).$ By
Lemma \ref{basico} (third bullet) there exists $t_{2}\in\mathbb{G}$ such that
$T^{t_{2}}z\in B_{\varepsilon/2}(y)$ and $d(T^{t_{2}}x,T^{t_{2}}%
z)\leq\varepsilon/2;$ thus $T^{t_{2}}x\in B_{\varepsilon}(y).$ This means the
system is minimal.
\end{proof}

A similar result is known for a null system ( Definition
\ref{topologicalsequence}), i.e.\ every Banach transitive null system is
minimal \cite{huang2003null}. It is an open question whether every transitive
null $\mathbb{Z}_{+}-$system is minimal (see \cite{huang2003null} and
\cite{glasner2009local}).

\section{Measure theoretical results}

Measure theoretical equicontinuity for TDS with respect to Borel probability
measures has been studied in ~\cite{Gilman1}, \cite{Cadre2005375},
\cite{mtequicontinuity} and \cite{mueqtds}. A natural question is to ask how
this concept relates to other known forms of rigidity for ergodic systems, for
example discrete spectrum (see Definition \ref{ds}). In
\cite{mtequicontinuity} it was shown that $\mu$-equicontinuous systems have
discrete spectrum. However, the converse is not true; for example Sturmian and
regular Toeplitz subshifts (equipped with their unique ergodic measure) are
not $\mu-$equicontinuous but have discrete spectrum.

In this section we introduce $\mu-$mean equicontinuity (a measure theoretical
form of mean equicontinuity) and $\mu-$mean sensitivity. The main result of
this section states that an ergodic TDS has discrete spectrum if and only if
it is $\mu-$mean equicontinuous if and only if it is not $\mu-$mean sensitive
(Corollary \ref{principal}).

\bigskip A $\mathbb{G}$\textbf{\ -measure preserving transformation
($\mathbb{G-}$MPT)} is a triplet $(M,\mu,T)$ where $(M,\mu)$ is a standard
probability space and $T:=\left\{  T^{i}:i\in\mathbb{G}\right\}  $ is a
$\mathbb{G}-$ measure preserving action on $M.$ When we say a system is
ergodic we mean that it is measure preserving and ergodic.

\subsection{$\bigskip\mu-$Mean equicontinuity}

We denote Borel probability measures of $X$ by $\mu$ and we define
$\mathcal{B}_{X}^{+}:=\left\{  A\text{ is Borel}:\mu(A)>0\right\}  .$

\begin{definition}
Let $(X,T)$ be a TDS and $\mu$ a (not necessarily invariant) Borel probability
measure on $X.$ We say $(X,T)$ is $\mu-$\textbf{mean equicontinuous} if for
every $\kappa>0$ there exists a compact set $M$ such that $\mu(M)>1-\kappa$
and $T\shortmid_{M}$is mean equicontinuous i.e. for every $x\in M$ and every
$\varepsilon>0$ there exists $\delta>0$ such that if $y\in B_{\delta}(x)\cap
M$ then $\overline{D}(i\in\mathbb{G}:d(T^{i}x,T^{i}y)>\varepsilon
)<\varepsilon$ (this implies that $f_{b}\shortmid_{M}$ is continuous)$.$
\end{definition}

\begin{remark}
By approximation arguments we could equivalently define $\mu-$mean
equicontinuity by asking that $M$ is simply Borel (and not necessarily compact).
\end{remark}

This definition may remind the reader of Lusin's theorem. In fact we can use
this to obtain information about $\mu-$mean equicontinuous systems.

\begin{definition}
Let $X$ be a compact metric space, $\mu$ a Borel probability measure on $X$,
and $Y$ a metric space$.$

A set $A\subset X$ is $\mu-$\textbf{measurable} if $A$ is in the sigma-algebra
generated by the completion of $\mu.$

A function $f:X\rightarrow Y$ $\ $is $\mu-$\textbf{measurable }if for every
Borel set $B$ we have that $f^{-1}(B)$ is $\mu-$measurable.

A function $f:X\rightarrow Y$ is $\mu-$\textbf{Lusin (}or \textbf{Lusin
measurable}) if for every $\kappa>0$ there exists a compact set $M\subset X$
such that $\mu(M)>1-\kappa$ and $f\mid_{M}$ is continuous.
\end{definition}

\bigskip It is not difficult to see that $(X,T)$ is $\mu-$mean equicontinuous
if and only if $f_{b}$ (see Definition \ref{besi}) is $\mu-$Lusin.

Every $\mu-$Lusin function is $\mu-$measurable. The converse is true if $Y$ is
separable (Lusin's theorem); this fact is generalized in the following result.

\begin{theorem}
[Lusin's theorem \cite{simonnet1996measures} pg. 63/145]\label{simon}Let $X$
be a compact metric space, $\mu$ a Borel probability measure on $X$, $Y$ a
metric space$,$ and $f:X\rightarrow Y$ be a function such that there exist
$X^{\prime}\subset X$ such that $\mu(X^{\prime})=1$ and $f(X^{\prime})$ is
separable. We have that $f$ is $\mu-$Lusin if and only if for every open ball
$B,$ $f^{-1}(B)$ is $\mu-$measurable.
\end{theorem}

\begin{remark}
\label{este}Since
\[
d_{b}(x,y)=\inf\left\{  \delta>0:\limsup_{n\rightarrow\infty}\frac{\left\vert
\left\{  i\in\mathbb{G\mid}\text{ }d\mathbb{(}T^{i}x,T^{i}y)>\delta\right\}
\cap F_{n}\right\vert }{\left\vert F_{n}\right\vert }<\delta\right\}  ,
\]
and $\mu$ is Borel, $d_{b}(x,y)$ is a Borel function. This implies that for
every $\varepsilon>0$ and every $x\in X$ we have that $B_{\varepsilon}^{b}(x)$
is $\mu-$measurable.
\end{remark}

\begin{proposition}
\label{separ}Let $(X,T)$ be a TDS and $\mu$ a Borel probability measure. We
have that $(X,T)$ is $\mu-$mean equicontinuous if and only if there exists
$X^{\prime}\subset X$ such that $\mu(X^{\prime})=1$ and $(X^{\prime}\diagup
d_{b},d_{b})$ is separable.
\end{proposition}

\begin{proof}
Define $f:=f_{b}$.

If there exists $X^{\prime}\subset X$ such that $\mu(X^{\prime})=1$ and
$(X^{\prime}\diagup d_{b},d_{b})$ is separable apply Theorem \ref{simon} to
obtain that $f_{b}$ is $\mu-$Lusin and hence $(X,T)$ is $\mu-$mean equicontinuous.

If $f_{b}$ is $\mu-$Lusin it means that for every $\kappa>0$ there exists a
compact set $M_{\kappa}\subset X$ such that $\mu(M_{\kappa})>1-\kappa$ and
$f_{b}\mid_{M_{\kappa}}$ is continuous. This implies that $X^{\prime}%
=\cup_{n\in\mathbb{N}}M_{1/n}$ $\ $satisfies the desired conditions.
\end{proof}

Under some circumstances we can describe $\mu-$mean equicontinuous systems
using $\mu-$mean equicontinuity points.

\begin{definition}
\bigskip We say $x\in X$ is a $\mu-$\textbf{mean equicontinuous point} if for
every $\varepsilon>0$%
\[
\lim_{\delta\rightarrow0}\frac{\mu(B_{\delta}(x)\cap B_{\varepsilon}^{b}%
(x))}{\mu(B_{\delta}(x))}=1.
\]

\end{definition}

We can apply Theorem 16 in \cite{mueqtds} to obtain the following result.

\begin{theorem}
\label{meaneq}Let $(X,T)$ be a TDS and $\mu$ a Borel probability measure.
Consider the following properties.

1)$(X,T)$ is $\mu-$mean equicontinuous.

2) Almost every $x\in X$ is a $\mu-$mean equicontinuous point.

If $(X,\mu)$ satisfies the Lebesgue density theorem then $1)\Longrightarrow
2).$

If $(X,\mu)$ is Vitali (i.e. satisfies Vitali's covering theorem) then
$2)\Longrightarrow1).$
\end{theorem}

If $X$ is a Cantor space or $X\subset\mathbb{R}^{d}$ then for every Borel
measure $\mu,$ $(X,\mu)$ is Vitali and satisfies the Lebesgue density theorem.
For more information see \cite{mueqtds}.

Measure theoretic equicontinuity can been studied under non invariant
measures; for example the existence and ergodicity of limit measures of $\mu
-$equicontinuous cellular automata was studied in \cite{mueqca}. From now on
we will only study measure preserving systems.

\subsection{$\mu-$Mean sensitivity}

Measure theoretic forms of sensitivity have been studied in \cite{Gilman1},
\cite{Cadre2005375} and \cite{mtequicontinuity}. In particular in
\cite{mtequicontinuity} it was shown that ergodic TDS are either $\mu
-$equicontinuous or $\mu-$sensitive.

We also show that $\mu-$mean equicontinuity is a measurable invariant property
for TDS (this is not satisfied by $\mu-$equicontinuous TDS).

\begin{definition}
A TDS $(X,T)$ is $\mu-$\textbf{mean sensitive} if there exists $\varepsilon>0$
such that for every $A\in\mathcal{B}_{X}^{+}$ there exists $x,y\in A$ such
that $d_{b}(x,y)>\varepsilon$ (and hence $\overline{D}\left\{  i\in
\mathbb{G}:d(T^{i}x,T^{i}y)>\varepsilon\right\}  >\varepsilon)$. In this case
we say $\varepsilon$ is a $\mu-$\textbf{mean sensitivity constant. }
\end{definition}

\begin{definition}
\label{expansive}A TDS $(X,T)$ is $\mu-$\textbf{mean expansive} if there
exists $\varepsilon>0$ such that $\mu\times\mu\left\{  (x,y):d_{b}%
(x,y)>\varepsilon\right\}  =1.$
\end{definition}

The following fact is well known. We give a proof for completeness.

\begin{lemma}
\label{separable}Let $\mathcal{(}Y,d_{Y})$ be a metric space$.$ Suppose that
there is no uncountable set $A\subset Y$ and $\varepsilon>0$ such that
$d_{Y}(x,y)>\varepsilon$ for every $x,y\in A$ with $x\neq y,$ then $(Y,d_{Y})
$ is separable.
\end{lemma}

\begin{proof}
For every $\varepsilon>0$ rational we define:%
\[
\mathcal{F}_{\varepsilon}:=\left\{  A\subset Y:d_{Y}(x,y)>\varepsilon\text{
}\forall x\neq y\in A\right\}  .
\]

Using Zorn's lemma we obtain that $\mathcal{F}_{\varepsilon}$ admits a maximal
element $M_{\varepsilon}$, which by hypothesis must be countable. Then
$M:=\cup_{\varepsilon\in\mathbb{Q}_{+}}M_{\varepsilon}$ is also countable. We
have that for every $x\in X$ and $\varepsilon>0$ there exists $y\in M$ such
that $d_{Y}(x,y)\leq\varepsilon.$
\end{proof}

\begin{lemma}
\label{ecu}Let $(X,\mu,T)$ be an ergodic TDS. Then $f(x):=\mu(B_{\varepsilon
}^{b}(x))\ $is constant for almost every $x\in X$ and equal to $\mu\times
\mu\left\{  (x,y):d_{b}(x,y)\leq\varepsilon\right\}  .$
\end{lemma}

\begin{proof}
By Remark \ref{este} $d_{b}(x,y)$ is $\mu\times\mu-$measurable. This means
that $\left\{  (x,y):d_{b}(x,y)\leq\varepsilon\right\}  $ is $\mu\times\mu
-$measurable for every $\varepsilon>0$. Using Fubini's Theorem we obtain that
\begin{align*}
\mu\times\mu\left\{  (x,y):d_{b}(x,y)<\varepsilon\right\}   &  =\int_{X}%
\int_{X}1_{\left\{  (x,y):d_{b}(x,y)\leq\varepsilon\right\}  }d\mu(y)d\mu(x)\\
&  =\int_{X}\mu\left\{  y:d_{b}(x,y)\leq\varepsilon\right\}  d\mu(x)\\
&  =\int_{X}\mu(B_{\varepsilon}^{b}(x))d\mu(x).
\end{align*}

Since $f$ is $T$-invariant we conclude that $f(x)$ is constant for almost
every $x\in X$ and equal to $\mu\times\mu\left\{  (x,y):d_{b}(x,y)<\varepsilon
\right\}  .$
\end{proof}

\begin{theorem}
\label{strongsensitive}Let $(X,\mu,T)$ be an ergodic TDS. The following are equivalent:

$1)$ $(X,T)$ is $\mu-$mean sensitive.

$2)$ $(X,T)$ is $\mu-$mean expansive$.$

$3)$ There exists $\varepsilon>0$ such that for almost every $x,$
$\mu(B_{\varepsilon}^{b}(x))=0.$

$4)(X,T)$ is not $\mu-$mean equicontinuous.
\end{theorem}

\begin{proof}
$2)\Rightarrow1)$

Let $A\in B_{X}^{+}.$ This means that $A\times A\in B_{X^{2}}^{+}.$ By
hypothesis we can find $(x,y)\in A\times A$ such that $\overline{D}\left\{
i:d(T^{i}x,T^{i}y)>\varepsilon\right\}  >\varepsilon.$

$1)\Rightarrow3)$

Suppose $(X,T)$ is $\mu-$mean sensitive (with $\mu-$mean sensitivity constant
$\varepsilon)$ and that $3)$ is not satisfied. This means there exists $x\in
X$ such that $B_{\varepsilon/3}^{b}(x)\in B_{X}^{+}.$ For any $y,z\in
B_{\varepsilon/3}^{b}(x)$ we have that $d_{b}(y,z)<\varepsilon.$ This
contradicts the assumption that $T$ is $\mu-$mean sensitive.

$3)\Rightarrow2)$

Using Lemma \ref{ecu} we obtain that $\mu\times\mu\left\{  (x,y):d_{b}%
(x,y)\leq\varepsilon\right\}  =0.$

$2)\Rightarrow4)$

If $(X,T)$ is $\mu-$mean expansive then there exists $\varepsilon>0$ such that
$\mu\times\mu\left\{  (x,y):d_{b}(x,y)>\varepsilon\right\}  =1.$ Suppose
$(X,T)$ is $\mu-$mean equicontinuous. This implies there exists a compact set
$M$ such that $\mu(M)$ $>0$ and $\left.  f_{b}\right\vert _{M}$ is continuous
(and hence uniformly continuous). This implies there exists $\delta>0$ such
that if $x,y\in M$ and $d(x,y)\leq\delta$ then $d_{b}(x,y)\leq\varepsilon/2.$
We can cover $M$ with finitely many $\delta/2-$balls, this implies one of them
must have positive measure. Using this and $\mu-$mean expansiveness we would
obtain that there exists $p,q\in M$ such that $d(p,q)\leq\delta$ and
$d_{b}(p,q)>\varepsilon;$ a contradiction to the continuity of $\left.
f_{b}\right\vert _{M}.$

$4)\Rightarrow3)$

Suppose $3)$ is not satisfied. By Lemma \ref{ecu} we have that for every
$n\in\mathbb{N}$ there exists a set of full measure $Y_{n}$ and $a_{n}>0$ such
that $\mu(B_{1/n}^{b}(x))=a_{n}$ for all $x\in Y_{n}.$ Let $Y:=\cap
_{n\in\mathbb{N}}Y_{n}.$ If $(Y\diagup d_{b},d_{b})$ is not separable then by
Lemma \ref{separable} there exists $\varepsilon>0$ and an uncountable set $A$
such that for every $x,y\in A$ such that $x\neq y$ we have that
$B_{\varepsilon}^{b}(x)\cap B_{\varepsilon}^{b}(y)=\emptyset.$ This is a
contradiction, hence $(Y\diagup d_{b},d_{b})$ is separable. Using Proposition
\ref{separ} we conclude $(X,T)$ is $\mu-$mean equicontinuous.
\end{proof}

\begin{definition}
Two $\ $measure preserving transformations, $(M,\mu,T)$ and $(M^{\prime}%
,\mu^{\prime},T^{\prime}),$ are\textbf{\ isomorphic (measurably)} if there
exists an a.e. bijective and measure preserving function $f:(M,\mu
)\rightarrow(M^{\prime},\mu^{\prime})$ such that $T^{\prime i}\circ f=f\circ
T^{i}$ for every $i\in\mathbb{G}$ and $f^{-1}$ is also measure preserving.

We say $(M^{\prime},\mu^{\prime},T^{\prime})$ is a \textbf{factor} of
$(M,\mu,T)$ if there exists a surjective and measure preserving function
$f:(M,\mu)\rightarrow(M^{\prime},\mu^{\prime})$ such that $T^{\prime i}\circ
f=f\circ T^{i}$ for every $i\in\mathbb{G}$ .
\end{definition}

\begin{proposition}
\label{isomorphism}Let $(X,\mu,T)$ and $(X^{\prime},\mu^{\prime},T^{\prime})$
be two ergodic topological systems. If $(X,T)$ is $\mu-$mean equicontinuous
and $(X^{\prime},\mu^{\prime},T^{\prime})$ is a factor of $(X,\mu,T)$ then
$(X^{\prime},T^{\prime})$ is $\mu^{\prime}-$mean equicontinuous.
\end{proposition}

\begin{proof}
We will denote by $d$ and $d^{\prime}$ the metrics of $X$ and $X^{\prime}$ respectively.

Suppose $(X^{\prime},T^{\prime})$ is not $\mu^{\prime}-$mean equicontinuous.
By the previous Theorem we have that $(X^{\prime},T^{\prime})$ is $\mu
^{\prime}-$mean expansive, i.e. there exists $\varepsilon>0$ and a set
$Y^{\prime}\subset X^{\prime}\times X^{\prime}$such that for every
$(x^{\prime},y^{\prime})\in Y^{\prime}$ we have that $d_{b}(x^{\prime
},y^{\prime})>\varepsilon$ and $\mu^{\prime}\times\mu^{\prime}(Y^{\prime})=1.$

Let $f:X\rightarrow X^{\prime}$ be the factor map. By Lusin's Theorem we know
that there exists a compact set $K\subset X$ such that $\mu(K)\geq
1-\varepsilon/4$ and $f\mid_{K}$ is continuous. This implies there exists
$\varepsilon_{1}>0$ such that if $f(x),f(y)\in f(K)$ and $d^{\prime
}(f(x),f(y))>\varepsilon$ then $d(x,y)>\varepsilon_{1}.$ This implies
$\overline{D}\left\{  i\in\mathbb{G}:d^{\prime}(T^{^{\prime}i}f(x),T^{^{\prime
}i}f(y)\text{ }>\varepsilon\right\}  \geq\varepsilon.$ We define
\begin{align*}
E_{1}(x,y)  &  :=\left\{  i\in\mathbb{G}:T^{i}x,T^{i}y\in K\text{ }\right\}
,\text{ and }\\
E_{2}(x,y)  &  :=\left\{  i\in\mathbb{G}:d(T^{i}x,T^{i}y)\text{ }%
>\varepsilon_{1}\right\}  .
\end{align*}
Using that $\mu(K)\geq1-\varepsilon/4$ and the ergodic theorem we have that
for almost every $x,y\in X$ we have that $(f(x),f(y))\in Y^{\prime}$ and
$\underline{D}(E_{1}(x,y))\geq1-\varepsilon/2.$ Since $\left\{  i\in
\mathbb{G}:d^{\prime}(T^{^{\prime}i}f(x),T^{^{\prime}i}f(y)\text{
}>\varepsilon\right\}  \subset E_{2}(x,y)$ we have that $\overline{D}%
(E_{2}(x,y))\geq\varepsilon.$ This implies that for a.e. $x,y\in X$ we have
that $d(T^{i}x,T^{i}y)>\varepsilon_{1}$ for every $i\in E_{1}(x,y)\cap
E_{2}(x,y),$ and that $\overline{D}(E_{1}(x,y)\cap E_{2}(x,y))\geq
\varepsilon/2.$ Hence $(X,\mu,T)$ is $\mu-$mean expansive (hence $\mu-$mean sensitive).
\end{proof}

\subsection{$\mu-$Mean sensitive pairs}

The notion of entropy pairs was introduced in \cite{blanchard1993disjointness}%
. Different kinds of pairs have been studied, in particular sequence entropy
pairs in \cite{Huang2004} and $\mu-$sensitive pairs \cite{mtequicontinuity}.
In \cite{mtequicontinuity} $\mu-$sensitive pairs were used to characterize
$\mu-$sensitivity; we introduce $\mu-$mean sensitive pairs to characterize
$\mu-$mean sensitivity.

\begin{definition}
[\cite{Huang2004}]Let $(X,\mu,T)$ be a measure preserving TDS. We say $(x,y)$
is a $\mu-$sequence entropy pair if for any finite partition $\mathcal{P},$
such that there is no $P\in\mathcal{P}$ such that $x,y\in\overline{P},$ there
exists $S\subset\mathbb{G}$ such that $h_{\mu}^{S}(\mathcal{P},T)>0.$
\end{definition}

The following result was proven for $\mathbb{Z}_{+}-$systems in
\cite{Huang2004} (Theorem 4.3), and was generalized in \cite{KerrMeasurable}
(Proposition 4.7 and Theorem (2) and Proposition 4.9)

\begin{theorem}
\label{seqpairs}An ergodic TDS $(X,\mu,T)$ is $\mu-$null if and only if there
are no $\mu-$sequence entropy pairs.
\end{theorem}

\begin{definition}
We say $(x,y)\in X^{2}$ is a $\mu-$\textbf{mean sensitive pair }if $x\neq y$
and for all open neighbourhoods $U_{x}$ of $x$ and $U_{y}$ of $y,$ there
exists $\varepsilon>0$ such that for every $A\in\mathcal{B}_{X}^{+}$\ there
exist $p,q\in A$ with $\overline{D}(i\in\mathbb{G}:T^{i}p\in U_{x}$ and
$T^{i}q\in U_{y})>\varepsilon.$ We denote the set of $\mu-$mean sensitive
pairs as $S_{\mu}^{m}(X,T).$
\end{definition}

\begin{theorem}
\label{sensipairs}Let $(X,\mu,T)$ be an ergodic TDS. Then $S_{\mu}%
^{m}(X,T)\neq\emptyset$ if and only if $(X,T)$ is $\mu-$mean sensitive.
\end{theorem}

\begin{proof}
If $(x,y)\in S_{\mu}^{m}(X,T)$ then there exists open neighbourhoods $U_{x}$
of $x$ and $U_{y}$ of $y$ (with $d(U_{x},U_{y})>0)$ and $\varepsilon>0$ such
that for every $A\in\mathcal{B}_{X}^{+}$\ there exist $p,q\in A$ and
$S\subset$ $\mathbb{G}$ with $\overline{D}(S)>\varepsilon$ such that
$T^{i}p\in U_{x}$ and $T^{i}q\in U_{y}$ for every $i\in S.$ This implies that
$\overline{D}\left\{  i\in\mathbb{G}:d(T^{i}x,T^{i}y)\geq d(U_{x}%
,U_{y})\right\}  >\varepsilon.$ Thus $(X,T)$ is $\mu-$mean sensitive.

Let $(X,T)$ be $\mu-$mean sensitive with sensitive constant $\varepsilon_{0}$
and $0<\varepsilon<\varepsilon_{0}$.

For $\varepsilon>0$ we define the compact set
\[
X^{\varepsilon}:=\left\{  (x,y)\in X^{2}\mid d(x,y)\geq\varepsilon\right\}  .
\]

Suppose that $S_{\mu}^{m}(X,T)=\emptyset.$ In particular this implies that for
every $(x,y)\in X^{\varepsilon}$ there exist open neighbourhoods of $x$ and
$y,$ $U_{x}$ and $U_{y}$, such that for every $\delta>0$ there exists a set
$A_{\delta}(x,y)\in\mathcal{B}_{X}^{+}$ such that
\[
\overline{D}\left\{  i\in\mathbb{G}:(T^{i}p,T^{i}q)\in U_{x}\times
U_{y}\right\}  \leq\delta
\]
for all $p,q\in A_{\delta}(x,y)$.

There exists a finite set of points $F\subset X^{\varepsilon}$ such that%
\[
X^{\varepsilon}\subseteq\cup_{(x,y)\in F}U_{x}\times U_{y}.
\]
Let $\delta=\varepsilon/\left\vert F\right\vert .$ Since $\mu$ is ergodic for
every $(x,y)\in F$ there exists $n(x,y)\in\mathbb{G}$ such that $A:=\cap
_{(x,y)\in F}T^{n(x,y)}A_{\delta}(x,y)\in\mathcal{B}_{X}^{+}$.$\ $Thus for
every $(x,y)\in F$
\[
\overline{D}\left\{  i\in\mathbb{G}:(T^{i}p,T^{i}q)\in U_{x}\times
U_{y}\right\}  \leq\delta,
\]
for every $p,q\in A.$

Since $\varepsilon$ is smaller than a sensitive constant there exist $p,q\in
A$ such that
\begin{align*}
\overline{D}\left\{  i\in\mathbb{G}:(T^{i}p,T^{i}q)\in X^{\varepsilon
}\right\}   &  >\varepsilon,\text{ and hence}\\
\overline{D}\left\{  i\in\mathbb{G}:(T^{i}p,T^{i}q)\in\cup_{(x,y)\in F}%
U_{x}\times U_{y}\right\}   &  >\varepsilon.
\end{align*}
We have a contradiction since this means there exists $(x^{\prime},y^{\prime
})\in F$ such that
\[
\overline{D}\left\{  i:(T^{i}p,T^{i}q)\in U_{x^{\prime}}\times U_{y^{\prime}%
}\right\}  >\varepsilon/\left\vert F\right\vert =\delta.
\]

\end{proof}

\bigskip The relationship between entropy (and sequence entropy) and
independence was studied in \cite{KerrMeasurable}. The following result shows
there is a relationship between $\mu-$mean sensitive pairs and a different
kind of measure theoretical independence pairs.

\begin{lemma}
\label{nind}\bigskip Let $(X,\mu,T)$ be an ergodic $TDS.$ Suppose that $x\neq
y$ and that for all open neighbourhoods $U_{x}$ of $x$ and $U_{y}$ of $y$,
there exists $\delta>0$ such that for every $N$ there exists $S_{N}%
\subset\mathbb{G}$, with $\left\vert S_{N}\right\vert \geq N,$ such that for
all $s_{i},s_{j}\in S_{N}$ we have that $\mu(T^{-s_{i}}U_{x}\cap T^{-s_{j}%
}U_{y})>\delta.$ Then $(x,y)\in S_{\mu}^{m}(T).$\newline
\end{lemma}

\begin{proof}
Let $A\in\mathcal{B}_{X}^{+}.$ There exist $N>0$ and $s_{1}\neq s_{2}\in
S_{N}$ such that
\[
\mu(T^{-s_{1}}A\cap T^{-s_{2}}A)>0.
\]

Let $W:=T^{-s_{1}}U_{x}\cap T^{-s_{2}}U_{y}.$ By the pointwise ergodic theorem
there exists a point $z\in T^{-s_{1}}A\cap T^{-s_{2}}A$ such that
$\underline{D}(S)=\mu(W)>\delta,$ where $S=\left\{  i>s_{1},s_{2}\mid
T^{i}z\in W\right\}  .$ Let $p:=T^{s_{1}}z$ and $q:=T^{s_{2}}z.$ We have that
$p,q\in A$ and $T^{i}p\in U_{x}$ and $T^{i}q\in U_{y}$ for every $i\in S.$This
means that $(x,y)\in S_{\mu}^{m}(T).$
\end{proof}

\subsection{\bigskip Discrete spectrum and sequence entropy}

3.4.1 {\large Sequence entropy}

\begin{definition}
[\cite{0036-0279-22-5-R02}]Let $(M,\mu,T)$ be a measure preserving
transformation. Given a finite measurable partition $\mathcal{P}$ of $X$ and
$S=\left\{  s_{n}\right\}  \subset\mathbb{G}$ we define $h_{\mu}%
^{S}(\mathcal{P},T):=\lim\sup_{n\rightarrow\infty}\frac{1}{n}H(\vee_{i=1}%
^{n}T^{-s_{i}}\mathcal{P)},$ and the \textbf{sequence entropy of }$(X,\mu
,T)$\textbf{\ with respect to }$S$ as $h_{\mu}^{S}(T):=\sup_{\mathcal{P}\text{
}}h_{\mu}^{S}(\mathcal{P},T)$. The system is said to be $\mu$\textbf{-null}
(or \textbf{zero sequence entropy}) if $h_{\mu}^{S}(T)=0$ for every
$S\subset\mathbb{G}.$
\end{definition}

The following remarkable lemma by Kushnirenko provides a connection between
entropy and functional analysis (which is in part responsible for the
connection between sequence entropy and discrete spectrum see section 3.4.2).

\begin{lemma}
[\cite{0036-0279-22-5-R02}]\label{Kushnirenko}Let $(M,\mu)$ be a probability
space and $\left\{  \xi_{n}\right\}  $ be a sequence of two-set partitions of
$M$, with $\xi_{n}=(P_{n},P_{n}^{c}).$ The closure of $\left\{  1_{P_{1}%
},1_{P_{2}}...\right\}  \subset L^{2}(M)$ is compact if and only if for all
subsequences%
\[
\lim_{n\rightarrow\infty}\frac{1}{n}H(\bigvee\limits_{i=1}^{n}\xi_{m_{i}})=0.
\]

\end{lemma}

Given a measure preserving system and a non-trivial measurable partition
$\left\{  B,B^{c}\right\}  $ we can associate a shift invariant measure
$\mu_{B}$ on $\left\{  0,1\right\}  ^{\mathbb{G}}$ as follows. We first define
the function $\phi_{B}:X\rightarrow\left\{  0,1\right\}  ^{\mathbb{G}}$ with
$(\phi(x))_{i}=0$ if and only if $T^{i}x\in B$ and we define $\mu_{B}=\phi
_{B}\mu.$

\begin{theorem}
\label{main}Let $(X,\mu,T)$ be an ergodic TDS. If $(X,T)$ is $\mu-$mean
equicontinuous then it is $\mu-$null.
\end{theorem}

\begin{proof}
We will show there exists a factor map $\phi_{B}:(X,\mu,T)\rightarrow(\left\{
0,1\right\}  ^{\mathbb{G}},\mu_{B},\sigma)$ such that $(\left\{  0,1\right\}
^{\mathbb{G}},\sigma)$ is not $\phi_{B}\mu-$mean equicontinuous. By
Proposition \ref{isomorphism} we conclude $(X,T)$ cannot be $\mu-$mean equicontinuous.

Since the system is not $\mu-$null there exists a two set partition
$\mathcal{P=(}B,B^{c})$ and sequence $S=\left\{  s_{n}\right\}  \subset
\mathbb{G}$ such that $h_{\mu}^{S}(\mathcal{P},T)>0.$

Now we consider the factor $(\left\{  0,1\right\}  ^{\mathbb{Z}},\mu
_{B},\sigma).$

Define the partition $\xi_{i}\mathcal{=(}\left\{  x_{i}=0\right\}  ,\left\{
x_{i}=1\right\}  ).$ By Lemma \ref{Kushnirenko} we have that the closure of
$\left\{  1_{\left\{  x_{s_{1}}=0\right\}  },1_{\left\{  x_{s_{2}}=0\right\}
}...\right\}  $ is not compact, and hence it is not totally bounded. So there
exists $\varepsilon>0$ such that for every $N\in\mathbb{N}$ there exists
$S_{N}\subset S$ with $\left\vert S_{N}\right\vert =N$ with
\begin{equation}
\mu(\left\{  x_{i}=0\right\}  \bigtriangleup\left\{  x_{j}=0\right\}
)=\int\left\vert 1_{\left\{  x_{i}=0\right\}  }-1_{\left\{  x_{j}=0\right\}
}\right\vert ^{2}d\mu\geq\varepsilon\text{ } \label{l2}%
\end{equation}

for every $i\neq j\in S^{\prime}.$

Now we want to show $(\left\{  0,1\right\}  ^{\mathbb{Z}},\sigma)$ is $\mu
_{B}-$mean sensitive. Let $A\in\mathcal{B}_{\left\{  0,1\right\}
^{\mathbb{Z}}}^{+}$ and $N$ sufficiently large$.$ There exist $s\neq t\in
S_{N}$ such that
\[
\mu(\sigma^{-s}A\cap\sigma^{-t}A)>0.
\]

Let $W:=\left\{  x:x_{s}\neq x_{t}\right\}  $ and $k\in\mathbb{N}$ such that
$s,t\in F_{k}.$ By the pointwise ergodic theorem there exists $z\in\sigma
^{-s}A\cap\sigma^{-t}A$ such that $\underline{D}(S_{W})=\mu(W),$ where
\[
S_{W}:=\left\{  i\notin F_{k}\mid\sigma^{i}z\in W\right\}  .
\]
By (\ref{l2}) we have that $\mu_{B}(W)>\varepsilon;$ hence $\underline{D}%
(S_{W})>\varepsilon.$ Let $p:=\sigma^{s}z$ and $q:=\sigma^{t}z.$ Since
$\mathbb{G}$ is commutative we have that $p_{i}\neq$ $q_{i}$ for every $i\in
S_{W}$.

We conclude that for every $A\in\mathcal{B}_{\left\{  0,1\right\}
^{\mathbb{Z}}}^{+}$ there exist $p,q\in A$ such that $\overline{D}\left\{
i\in\mathbb{G}:p_{i}\neq\text{ }q_{i}\right\}  >\varepsilon.$ This means that
$(\left\{  0,1\right\}  ^{\mathbb{Z}},\sigma)$ is $\mu_{B}-$mean sensitive,
hence not $\mu_{B}-$mean equicontinuous$.$
\end{proof}

3.4.2 {\large Discrete spectrum}

\bigskip

A measure preserving transformation is a \textbf{measurable isometry }if it is
isomorphic to an isometry (a TDS that satisfies $d(x,y)=d(T^{i}x,T^{i}y)$ for
all $i$).

A measure preserving transformation on a probability space $(M,\mu,T)$
generates a unitary linear operator on the Hilbert space $L^{2}(M,\mu),$ by
$U_{T}:f\mapsto f\circ T.$

\begin{definition}
\label{ds} We say $(M,\mu,T)$ has \textbf{discrete spectrum} if there exists
an orthonormal basis for $L^{2}(M,\mu)$ which consists of eigenfunctions of
$U_{T}.$
\end{definition}

The Halmos-Von Neumann theorem states ergodic system is a measurable isometry
if and only if it has discrete spectrum \cite{halmos1942operator}.

Kushnirenko's theorem states that an ergodic system is $\mu-$null if and only
if it has discrete spectrum \cite{0036-0279-22-5-R02}.

A $\mathbb{G-}$measure preserving transformation on a probability space
$(M,\mu,T)$ generates a family of unitary linear operator on the Hilbert space
$L^{2}(M,\mu),$ by $U_{T^{i}}:f\mapsto f\circ T^{i}.$ We say $(M,\mu,T)$ has
\textbf{discrete spectrum} if $L^{2}(M,\mu)$ is the direct sum of finite
dimensional $U_{T}$-invariant subspaces. Mackey showed that Halmos and Von
Neumann's Theorem holds for locally compact group actions
\cite{mackey1964ergodic}.

Kushnirenko's result was generalized for discrete actions by Kerr-Li in
\cite{KerrMeasurable}.

\begin{theorem}
\label{discrete}Let $(X,\mu,T)$ be an ergodic TDS. If $(X,\mu,T)$ has discrete
spectrum then it is $\mu-$mean equicontinuous.
\end{theorem}

\begin{proof}
This follows from the Halmos-Von Neumman Theorem and Proposition
\ref{isomorphism}.
\end{proof}

We have seen several characterizations of $\mu-$mean equicontinuity throughout
the paper we now chose the ones we consider the most interesting.

\begin{corollary}
\label{principal}Let $(X,\mu,T)$ be an ergodic TDS. Then $(X,\mu,T)$ satisfies
either a property on the left or on the right (which are all equivalent per
column):%
\[
\left.
\begin{array}
[c]{c}%
\mu-\text{mean equicontinuity}\\
\text{discrete spectrum}\\
\mu-\text{null}%
\end{array}
\right\vert
\begin{array}
[c]{c}%
\mu-\text{mean sensitivity}\\
\mu-\text{mean expansivity}\\
\text{there exists a }\mu-\text{mean sensitive pair}%
\end{array}
\]

\end{corollary}

\begin{proof}
Apply Theorem \ref{discrete}, Theorem \ref{strongsensitive}, Theorem
\ref{main}, Kushnirenko's theorem and Theorem \ref{sensipairs}.
\end{proof}

\bigskip In \cite{scarpellini1982stability}\ it was asked if every mean
equicontinuous equipped with an ergodic measure has discrete spectrum.
Corollary \ref{principal} implies the answer of this question is positive.
This question has been independently solved directly in \cite{li2013mean}
using different tools.

\bigskip

\subsection{Sensitivity for measure preserving tranformations}

Using our results we can develop a notion of sensitivity for purely measure
preserving tranformations.

A topological model of a measure preserving transformation is a TDS that is
isomorphic to the transformation.

\begin{definition}
Let $(M,\mu,T)$ be a $\mathbb{G-}$measure preserving transformation. We say
$(M,\mu,T)$ is\textbf{\ measurably sensitive} if there exists a topological
model that is $\nu-$mean sensitive (where $\nu$ is the image of $\mu$ under
the isomorphism).
\end{definition}

\begin{theorem}
\label{measuret}Let $(M,\mu,T)$ be an ergodic $\mathbb{G-}$measure preserving
transformation. The following conditions are equivalent

$1)(M,\mu,T)$ is measurably sensitive.

$2)$Every topological model is $\nu-$mean sensitive ($\nu$ is the image of
$\mu$ under the isomorphism).

$3)$Every minimal topological model is mean sensitive.

$4)$ $(M,\mu,T)$ does not have purely discrete spectrum.
\end{theorem}

\begin{proof}
$1)\Leftrightarrow2)$

Apply Proposition \ref{isomorphism} and Corollary \ref{principal}.

$2)\Leftrightarrow4)$

Apply Corollary \ref{principal}.

$3)\Rightarrow4)$

If $(M,\mu,T)$ has discrete spectrum then by Halmos- Von Neumann theorem there
exists a minimal equicontinuous topological model.

$4)\Rightarrow3)$

If a minimal topological model is not mean sensitive then it is mean
equicontinuous (by Theorem \ref{sensi}). This would imply the system is $\mu
-$mean equicontinuous and hence has discrete spectrum.
\end{proof}

\section{Topological results}

In this section we will see that the characterizations of Corollary
\ref{principal} do not hold in the topological setting.

It is known that equicontinuous systems have zero topological sequence entropy
with respect to every subsequence (also known as null systems, see Definition
\ref{topologicalsequence}) and that there exist sensitive null systems
\cite{goodman1974topological}.

In this section we define another weak form of equicontinuity (diam-mean
equicontinuity) and another strong form of sensitivity (diam-mean sensitivity).

We will see that for (not necessarily minimal) subshifts we have the following
picture:
\begin{align*}
\text{mean sensitivity}  &  \Longrightarrow\text{diam-mean sensitivity}\\
&  \Longrightarrow\text{not null}\\
&  \Longrightarrow\text{sensitivity.}%
\end{align*}
and that each implication is strict. The implications for minimal subshifts
are stated in the introduction of the paper.

We will also show that almost automorphic subshifts are regular if and only if
they are diam-mean equicontinuous (Theorem \ref{aa}). Since minimal null
systems are almost automorphic \cite{huang2003null} we obtain that minimal
null subshifts are regular almost automorphic (Corollary \ref{nullquasi}).

\subsection{Diam-mean equicontinuity and diam-mean sensitivity}

\begin{definition}
Let $A\subset X.$ We denote the diameter of $A$ as $diam(A).$
\end{definition}

We introduce the following notion.

\begin{definition}
\label{weakeq}Let $(X,T)$ be a TDS. We say $x\in X$ is a \textbf{diam-mean
equicontinuous point} if for every $\varepsilon>0$ there exists $\delta>0$
such that
\[
\overline{D}\left\{  i\in\mathbb{G}:diam(T^{i}B_{\delta}(x))>\varepsilon
\right\}  <\varepsilon.
\]
We say $(X,T)$ is \textbf{diam-mean equicontinuous }if every $x\in X$ is a
diam-mean equicontinuous point. We say $(X,T)$ is \textbf{almost diam-mean
equicontinuous }if the set of diam-mean equicontinuity points is residual.
\end{definition}

\begin{remark}
Equivalently $x\in X$ is a diam-mean equicontinuous point if for every
$\varepsilon>0$ there exists $\delta>0$ such that \underline{$D$}$\left\{
i\in\mathbb{G}:diam(T^{i}B_{\delta}(x))\leq\varepsilon\right\}  \geq
1-\varepsilon.$
\end{remark}

By adapting the proof that a continuous function on a compact space is
uniformly continuous one can show a TDS is diam-mean equicontinuous if and
only if it is uniformly diam-mean equicontinuous i.e. for every $\varepsilon
>0$ there exists $\delta>0$ such that $\overline{D}\left\{  i\in
\mathbb{G}:diam(T^{i}B_{\delta}(x))>\varepsilon\right\}  <\varepsilon.$

\begin{definition}
We denote the set of diam-mean equicontinuity points by $E^{w}$ and we define
\[
E_{\varepsilon}^{w}:=\left\{  x\in X:\exists\delta>0,\text{ s.t.
}\underline{D}\left\{  i\in\mathbb{G}:diam(T^{i}B_{\delta}(x))\leq
\varepsilon\right\}  \geq1-\varepsilon\right\}  .
\]

\end{definition}

Note that $E^{w}=\cap_{\varepsilon>0}E_{\varepsilon}^{w}.$

\begin{lemma}
\label{qinvariant}Let $(X,T)$ be a TDS. The sets $E^{w}$ and $E_{\varepsilon
}^{w}$ are inversely invariant (i.e. $T^{-j}(E^{w})\subseteq E^{w}$ and
$T^{-j}(E_{\varepsilon}^{w})\subseteq E_{\varepsilon}^{w}$ for all
$j\in\mathbb{G})$ and $E_{\varepsilon}^{w}$ is open.
\end{lemma}

\begin{proof}
Let $\varepsilon>0$ and $T^{j}x\in E_{\varepsilon}^{w}.$ There exists $\eta>0$
and $S\subset\mathbb{G}$ such that $\underline{D}(S)\geq1-\varepsilon\ $and
$d(T^{i+j}x,T^{i+j}z)\leq\varepsilon$ for every $i\in S$ and $z\in B_{\eta
}(T^{j}x).$ There exists $\delta>0$ such that if $d(x,y)<\delta$ then
$d(T^{j}x,T^{j}y)<\eta.$ We conclude that $x\in E_{\varepsilon}^{w}.$ Thus
$T^{-1}(E^{w})\subseteq E^{w}.$

Let $x\in E_{\varepsilon}^{w}$ and take $\delta>0$ a constant that satisfies
the property of the definition of $E_{\varepsilon}^{w}.$ If $d(x,w)<\delta/2$
then $w\in E_{\varepsilon}^{w}.$ Indeed if $y,z\in B_{\delta/2}(w)$ then
$y,z\in B_{\delta}(x).$
\end{proof}

\begin{definition}
A TDS $(X,T)$ is \textbf{diam-mean sensitive} if there exists $\varepsilon>0$
such that for every open set $U$ we have
\[
\overline{D}\left\{  i\in\mathbb{G}:diam(T^{i}U)>\varepsilon\right\}
>\varepsilon.
\]

\end{definition}

Other strong forms of "diameter" sensitivity have been studied in
\cite{0951-7715-20-9-006} the times of separation are taken to be cofinite
(complement is finite) or syndetic (bounded gaps).

The proof of the following result is analogous to that of Theorem \ref{sensi}
(using $E_{\varepsilon}^{w}$ instead of $E_{\varepsilon}^{m}$).

\begin{theorem}
\label{dichotomy}A transitive system is either almost diam-mean equicontinuous
or diam-mean sensitive. A minimal system is either diam-mean equicontinuous or
diam-mean sensitive.
\end{theorem}

\subsection{Almost automorphic systems}

\bigskip Let $(X_{1},T_{1})$ and $(X_{2},T_{2})$ be two TDSs and
$f:X_{1}\rightarrow X_{2}$ a continuous function such that $T_{2}^{i}\circ
f=f\circ T_{1}^{i}$ for every $i\in\mathbb{G}$.

If $f$ is surjective we say $f$ is a \textbf{factor map} and $(X_{2},T_{2})$
is a \textbf{factor} of $(X_{1},T_{1}).$ If $f$ is bijective we say $f$ is a
\textbf{conjugacy} and $(X_{1},T_{1})$ and $(X_{2},T_{2})$ are
\textbf{conjugate (topologically)}.

Any TDS $(X,T)$ has a unique (up to conjugacy) \textbf{maximal equicontinuous
factor}\textit{\ }i.e\textit{.} an equicontinuous factor\textit{\ }%
$f_{eq}:(X,T)\rightarrow(X_{eq},T_{eq})$ such that if $f_{2}:$
$(X,T)\rightarrow(X_{2},T_{2})$ is a factor map such that $(X_{2},T_{2})$ is
equicontinuous then there exists a factor map $g:(X_{eq},T_{eq})\rightarrow
(X_{2},T_{2})$ such that $g\circ f_{eq}=f_{2}.$ The equivalence relation whose
equivalence classes are the fibers of $f_{eq\text{ }}$is called the
\textbf{equicontinuous structure relation}. This relation can be characterized
using the regionally proximal relation. We say that $x,y\in X$ are
\textbf{regionally proximal }if there exist sequences $\left\{  t_{n}\right\}
_{n=1}^{\infty}\subset\mathbb{G}$ and $\left\{  x^{n}\right\}  _{n=1}^{\infty
},\left\{  y^{n}\right\}  _{n=1}^{\infty}\subset X$ such that
\[
\lim_{n\rightarrow\infty}x^{n}=x,\text{ }\lim_{n\rightarrow\infty}%
y^{n}=y,\text{ and }\lim_{n\rightarrow\infty}d(T^{t_{n}}x^{n},T^{t_{n}}%
y^{n})=0.
\]
The equicontinuous structure relation is the smallest closed equivalence
relation containing the regional proximal relation (\cite{auslander} Chapter 9).

For mean equicontinuous systems we can characterize the maximal equicontinuous
factor map using the Besicovitch pseudometric (Definition \ref{besi}).

\begin{proposition}
\label{feq}Let $(X,T)$ be mean equicontinuous. Then $f_{eq}=f_{b}.$
\end{proposition}

\begin{proof}
We have that $f_{b}$ is continuous and $(X/d_{b},d_{b},T)$ is equicontinuous
so if $f_{eq}(x)=f_{eq}(y)$ then $f_{b}(x)=f_{b}(y).$ If $f_{b}(x)=f_{b}(y),$
then $d_{b}(x,y)=0$ so there exists a sequence $\left\{  t_{n}\right\}  $ such
that $\lim_{n\rightarrow\infty}d(T^{t_{n}}x,T^{t_{n}}y)=0;$ hence $x$ and $y$
are regionally proximal, hence $f_{eq}(x)=f_{eq}(y).$
\end{proof}

A transitive equicontinuous system is conjugate to a system where $\mathbb{G}$
acts as translations on a compact metric abelian group. If $(X,T)$ is a
transitive TDS we denote the maximal equicontinuous factor by $G_{eq}$ (since
it is a group)$.$ The TDS $(G_{eq},T_{eq})$ has a unique ergodic invariant
probability measure, the normalized Haar measure on $G_{eq}$; this measure has
full support and will be denoted by $\nu_{eq}.$

\begin{definition}
\label{auto} We say a TDS is \textbf{almost automorphic} if it is an almost
1-1 extension of its maximal equicontinuous factor i.e. if $f_{eq}^{-1}%
f_{eq}(x)=\left\{  x\right\}  $ on a residual set. A transitive almost
automorphic TDS is \textbf{regular }if
\[
\nu_{eq}\left\{  g\in G_{eq}:f_{eq}^{-1}(g)\text{ is a singleton}\right\}
=1.
\]

\end{definition}

\bigskip It is not difficult to see that if there exists a transitive point
$x\in X$ such that $f_{eq}^{-1}f_{eq}(x)=\left\{  x\right\}  $ then $(X,T)$ is
almost automorphic. Transitive almost automorphic systems are minimal.

Two well known families of almost automorphic systems are the Sturmian
subshifts (maximal equicontinuous factor is an irrational circle rotation) and
Toeplitz subshifts (maximal equicontinuous factor is an odometer, see Section
4.4 for definition and examples).

An important class of TDS are the shift systems. Let $\mathcal{A}$ be a
compact metric space (with metric $d_{\mathcal{A}}$). For $x\in\mathcal{A}%
^{\mathbb{G}}$ and $i\in\mathbb{G}$ we use $x_{i}$ to denote the $i$th
coordinate of $x$ and
\[
\sigma:=\left\{  \sigma^{i}:\mathcal{A}^{\mathbb{G}}\rightarrow\mathcal{A}%
^{\mathbb{G}}\text{ }\mid(\sigma^{i}x)_{j}=x_{i+j}\text{ for all }%
x\in\mathcal{A}^{\mathbb{G}}\text{and }j\in\mathbb{G}\right\}
\]
to denote the \textbf{shift maps}$.$ Using the (Cantor) product topology
generated by the topology of $\mathcal{A}$, we have that $\mathcal{A}%
^{\mathbb{G}}$ is a compact metrizable space. A subset $X\subset
\mathcal{A}^{\mathbb{G}}$ is a \textbf{general shift system} if it is closed
and $\sigma-$invariant$;$ in this case $(X,\sigma)$ is a TDS. Every TDS is
conjugate to a general shift system (by mapping every point to its orbit).

A general shift system is a \textbf{subshift} if $\mathcal{A}$ is finite.

\begin{remark}
\label{qeq}Let $X\subset\mathcal{A}^{\mathbb{G}}$ be a general shift system.
We have that $(X,\sigma)$ is diam-mean equicontinuous if and only if for all
$\varepsilon>0$ there exists $\delta>0$ such that for all $x\in X$ there
exists $S\subset\mathbb{G}$ with \underline{$D$}$(S)\geq1-\varepsilon,$ such
that if $d(x,y)\leq\delta$ then $d_{\mathcal{A}}(x_{i},y_{i})\leq\varepsilon$
$\ $for all $i\in S.$
\end{remark}

\begin{definition}
\bigskip Let $(X,\sigma)$ be a transitive general shift system. We define
\[
\mathcal{D}:=\left\{  g\in G_{eq}:\exists x,y\in X\text{ such that }%
f_{eq}(x)=f_{eq}(y)=g\text{ and }x_{0}\neq y_{0}\right\}  ,
\]
where $x_{0}$ and $y_{0}$ represent the $0th$ coordinates (or the value at the
identity of the semigroup) of $x$ and $y$ respectively.
\end{definition}

In general $\mathcal{D}$ is first category. If $\mathcal{A}$ is finite and
$f_{eq}(x)\notin\mathcal{D}$ there exists a neighbourhood $U_{f_{eq}(x)}$ such
that if $f_{eq}(y)\in$ $U_{f_{eq}(x)}$ then $x_{0}=y_{0}$. So if $(X,\sigma)$
is an almost automorphic subshift then $\mathcal{D}$ is closed and nowhere
dense. Also note that $(X,\sigma)$ is regular if and only if $\nu
_{eq}(\mathcal{D})=0.$

\begin{lemma}
\label{lemmasub}Let $(X,\sigma)$ be a minimal almost automorphic subshift and
$F\subset\mathbb{G}$ a finite set. Then for all $g\in G_{eq},$ there exists
$k\in\mathbb{G}$ such that $T_{eq}^{k+i}(g)\notin D$ for all $i\in F$.
\end{lemma}

\begin{proof}
Let $\mathcal{D}^{\prime}:=\cup_{i\in F}T_{eq}^{-i}(\mathcal{D}).$ This means
that $\mathcal{D}^{\prime}$ is also closed and nowhere dense. Thus there
exists $k$ such that $T_{eq}^{k}(g)\in G_{eq}-\mathcal{D}^{\prime},$ hence
$T_{eq}^{k+i}(g)\notin\mathcal{D}$ for all $i\in F.$
\end{proof}

\begin{theorem}
\label{aa}Let $(X,\sigma)$ be a minimal almost automorphic subshift. Then
$(X,\sigma)$ is regular if and only if it is diam-mean equicontinuous.
\end{theorem}

\begin{proof}
Suppose that $\nu_{eq}(\mathcal{D})=0.$ Hence for every $\varepsilon>0$ there
exists $\delta>0$ such that if $U:=\left\{  g\mid d(g,\mathcal{D})\leq
\delta\right\}  $ then $\nu_{eq}(U)<\varepsilon.$

Let $x\in X.$ We define $S_{x}:=\left\{  i\in\mathbb{G}\mid T_{eq}^{i}%
f_{eq}x\notin U\right\}  ;$ hence \underline{$D$}$(S)=1-\nu(U)\geq
1-\varepsilon$ (using the pointwise ergodic theorem)$.$ For every $i\in S_{x}$
we have that $B_{\delta}(T_{eq}^{i}f_{eq}(x))\cap\mathcal{D=\emptyset}$. This
implies that if $d(T_{eq}^{i}f_{eq}(x),T_{eq}^{i}f_{eq}(y))\leq\delta$ then
$x_{i}=y_{i}.$ There exists $\delta^{\prime}>0$ such that if $d(x,y)\leq
\delta^{\prime}$ then $d(T_{eq}^{i}f_{eq}(x),T_{eq}^{i}f_{eq}(y))\leq\delta.$
We can now prove that $x$ is a diam-mean equicontinuous point. If
$d(x,y)\leq\delta^{\prime}$ then $d(T_{eq}^{i}f_{eq}(x),T_{eq}^{i}%
f_{eq}(y))\leq\delta$ for every $i,$ and hence $x_{i}=y_{i}$ for every $i\in
S_{x}.$ Using Remark \ref{qeq} we conclude $(X,T)$ is diam-mean equicontinuous.

Now suppose $(X,\sigma)$ is diam-mean equicontinuous. We have that
$f_{b}=f_{eq}$(Proposition \ref{feq}).

Let $\varepsilon>0$ and $\delta>0$ given by Remark \ref{qeq}.

Let $x\in X.$ Note that $\sigma^{i}(x)\in f_{eq}^{-1}(\mathcal{D})$ if and
only if there exists $y\in f_{b}^{-1}\circ f_{b}(x)$ such that $y_{i}\neq
x_{i}$ $.$

Using Lemma \ref{lemmasub} there exists $k\in\mathbb{G}$ such that if
$f_{eq}(x)=f_{eq}(y)$ then $d(\sigma^{k}y,\sigma^{k}x)\leq\delta.$

Since $(X,\sigma)$ is a diam-mean equicontinuous there exists $S\subset$
$\mathbb{G}$, with \underline{$D$}$(S)\geq1-\varepsilon,$ such that if
$d(\sigma^{k}x,\sigma^{k}y)\leq\delta$ then $d(x_{i},y_{i})\leq\varepsilon$
for all $i\in S.$

This means that for every $\varepsilon>0$
\[
\overline{D}\left\{  i:\exists y\in f_{b}^{-1}\circ f_{b}(x)\text{ s.t.
}d_{\mathcal{A}}(x_{i},y_{i})>\varepsilon\right\}  \leq\varepsilon.
\]

Using that
\[
\mathcal{D}=\cup_{n\in\mathbb{N}}\left\{  g\in G_{eq}:\exists x,y\text{ s.t
}f_{eq}(x)=f_{eq}(y)=g\text{ and }d_{\mathcal{A}}(x_{0},y_{0})>1/n\right\}  ,
\]
and the pointwise ergodic theorem we conclude that $\nu_{eq}(\mathcal{D})=0.$
\end{proof}

We do not know if every minimal diam-mean equicontinuous TDS\ is almost automorphic.

\begin{definition}
A TDS is \textbf{\underline{\textbf{diam-mean}} equicontinuous} if for every
$\varepsilon>0$ there exists $\delta>0$ such that for every $x\in X$ we have
\underline{$D$}$\left\{  i\in\mathbb{G}:diam(T^{i}B_{\delta}(x))>\varepsilon
\right\}  <\varepsilon.$
\end{definition}

It is clear that every diam-mean equicontinuous system is
\underline{diam-mean} equicontinuous. We do not know if the converse is true
in general. We will show that under some conditions they are equivalent.

\begin{definition}
A TDS $(X,T)$ is \underline{diam-mean} sensitive if there exists
$\varepsilon>0$ such that for every open set $U$ we have \underline{$D$%
}$\left\{  i\in\mathbb{G}:diam(T^{i}U))>\varepsilon\right\}  >\varepsilon.$
\end{definition}

The proof of the following theorem is analogous to the proofs of Theorems
\ref{sensi} and \ref{dichotomy}.

\begin{proposition}
\label{lowerd}A minimal TDS is either \underline{diam-mean} equicontinuous or
\underline{diam-mean} sensitive.
\end{proposition}

\begin{definition}
Let $(X,T)$ be a TDS and $x\in X$. We denote the orbit of $x$ with $o_{T}(x)$.
\end{definition}

\begin{proposition}
\label{y}Let $(X,\sigma)$ be a transitive almost automorphic subshift . The
function $h:\mathcal{D}^{c}\rightarrow\mathcal{A}$ defined as $h(g)=(f_{eq}%
^{-1}(g))_{0}$ is continuous and there exists $y\in X$ such that

$\cdot y$ is transitive.

$\cdot f_{eq}^{-1}f_{eq}(y)$ is a singleton, in other words $o_{\sigma_{eq}%
}(y)\cap\mathcal{D=\emptyset}.$
\end{proposition}

\begin{proof}
This follows from Theorem 6.4 in \cite{surveyodometers}.
\end{proof}

\begin{lemma}
Let $(X,\sigma)$ be a transitive almost automorphic subshift and
$w=a_{0}...a_{n-1}\in\mathcal{A}^{n}$ and $U=\left\{  x:x_{0}...x_{n-1}%
=w\right\}  \subset X$ a non-empty set . There exists $p\in U $ such that
$p^{\prime}=f_{eq}(p)$ is generic for $\mathcal{D}$ (with respect to $\nu
_{eq}$) and $\sigma_{eq}^{i}p^{\prime}\in\mathcal{D}^{c}$ for $i=0,...,n-1.$
\end{lemma}

\begin{proof}
Let $U_{a}=\left\{  g\in X_{eq}:g\notin\mathcal{D},(f_{eq}^{-1}(g))_{0}%
=a\right\}  $ and $h:\mathcal{D}^{c}\rightarrow\mathcal{A}$ the continuous
function from Proposition \ref{y}. This implies that for every $a,$ $U_{a}$ is
an open set. Hence $\cap_{i=0}^{n-1}\sigma_{eq}^{-i}U_{a_{i}}$ is an open set.

Let $y\in X$ be the point given by Proposition \ref{y}. Since $U$ is non-empty
and $y$ transitive, there exists $z\in o_{\sigma_{eq}}(y)\cap U.$ Considering
that $o_{\sigma_{eq}}(y)\cap\mathcal{D=\emptyset}$ we obtain $f_{eq}(z)\in$
$\cap_{i=0}^{n-1}\sigma_{eq}^{-i}U_{a_{i}}.$ Thus $\cap_{i=0}^{n-1}\sigma
_{eq}^{-i}U_{a_{i}}$ is a non-empty open set. Since $\nu_{eq\text{ }}$ is
fully supported it contains a generic point for $\mathcal{D}$.
\end{proof}

\begin{proposition}
\label{strongbar}Let $(X,\sigma)$ be a transitive almost automorphic subshift.
If $(X,\sigma)$ is not regular then it is \underline{diam-mean} sensitive.
\end{proposition}

\begin{proof}
Assume $(X,\sigma)$ is not regular. This means that $\nu_{eq}(\mathcal{D})>0.$

Let $w\in\mathcal{A}^{n}$ and $U=\left\{  x:x_{0}...x_{n-1}=w\right\}  \subset
X$ non-empty (these sets form a base of the topology)$.$ Let $p\in U$ be the
point given by the previous lemma. Let $S:=\left\{  i\in\mathbb{G}:\sigma
_{eq}^{i}p^{\prime}\in\mathcal{D}\right\}  .$ Since $p^{\prime}$ is generic
for $\mathcal{D}$ we have that \underline{$D$}$(S)=\nu_{eq}(\mathcal{D}).$
Furthermore for every $i\in S$ there exists $q\in X$ such that $f_{eq}%
(p)=f_{eq}(q)$ and $p_{i}\neq q_{i}.$ Since $\sigma_{eq}^{j}p^{\prime}%
\in\mathcal{D}^{c}$ for $j=0,...,n-1$ we have that $q\in U.$ Hence
$(X,\sigma)$ is \underline{diam-mean} sensitive.
\end{proof}

\begin{corollary}
\label{regularity}Let $(X,\sigma)$ be a minimal almost automorphic subshift.
The following conditions are equivalent:

$1)(X,\sigma)$ is diam-mean equicontinuous;

$2)(X,\sigma)$ is not diam-mean sensitive;

$3)(X,\sigma)$ is regular;

$4)(X,\sigma)$ is \underline{diam-mean} equicontinuous;

$5)(X,\sigma)$ is not \underline{diam-mean} sensitive.
\end{corollary}

\begin{proof}
Apply Theorem \ref{dichotomy} to get $1)\Leftrightarrow2).$

Apply Theorem \ref{aa} to get $2)\Leftrightarrow3).$

By definition $1)\Rightarrow4).$

Proposition \ref{strongbar} implies $5)\Rightarrow3).$

Proposition \ref{lowerd} implies $4)\Leftrightarrow5).$
\end{proof}

We do not know if this result holds in general for TDS.

\subsection{Topological sequence entropy}

Let $\mathcal{U\ }$and $\mathcal{V}$ be two open covers of $X.$ We define
$\mathcal{U\ }\vee\mathcal{V}:=\left\{  U\cap V:U\in\mathcal{U\ }%
,V\in\mathcal{V}\right\}  $ and $N(\mathcal{U})$ as the minimum cardinality of
a subcover of $\mathcal{U}.$

\begin{definition}
[\cite{goodman1974topological}]\label{topologicalsequence}Let $(X,T)$ be a
TDS, $S=\left\{  s_{m}\right\}  _{m=1}^{\infty}\subset\mathbb{G}$, and
$\mathcal{U}$ an open cover. We define%

\[
h_{top}^{S}(T,\mathcal{U)}:=\lim_{n\rightarrow\infty}\sup\frac{1}{n}\log
N(\vee_{m=1}^{n}T^{-s_{m}}(\mathcal{U))}.
\]
\newline

The \textbf{topological entropy along the sequence S} is defined by%
\[
h_{top}^{S}(T\mathcal{)}:=\sup_{\text{open covers }\mathcal{U}}h_{top}%
^{S}(T,\mathcal{U)}
\]

A TDS is \textbf{null} if the topological entropy along every sequence is zero.
\end{definition}

\begin{lemma}
\label{abc}Let $K$ be a finite set, $\varepsilon>0,$ and $h:K\rightarrow
2^{\mathbb{G}}$ such that $\underline{D}(h(k))>\varepsilon$ for every $k\in
F.$ There exist $K^{\prime}\subset K$ and $i\in\mathbb{G}$ with $\left\vert
K^{\prime}\right\vert \geq\varepsilon\left\vert K\right\vert /2$ such that
$i\in h(k)$ for every $k\in K^{\prime}.$
\end{lemma}

\begin{proof}
There exists $n_{0}\in\mathbb{N}$ such that%
\[
\frac{\left\vert h(k)\cap F_{n_{0}}\right\vert }{\left\vert F_{n_{0}%
}\right\vert }\geq\varepsilon/2\text{ }\ \ \text{for every }k\in K.
\]

This means that
\[
\sum_{j\in F_{n_{0}}}\left\vert k\in K:j\subset h(k)\right\vert =\sum
_{k\in\mathcal{K}}\left\vert j\in F_{n_{0}}:j\in h(k)\right\vert \geq
\frac{\varepsilon}{2}\left\vert K\right\vert \left\vert F_{n_{0}}\right\vert
.
\]
Hence there exists $j\in F_{n_{0}}$ such that $\left\vert k\in K:j\subset
h(k)\right\vert \geq\frac{\varepsilon}{2}\left\vert K\right\vert .$
\end{proof}

\begin{theorem}
\label{seqentr}Let $(X,T)$ be a TDS. If $(X,T)$ is \ \underline{diam-mean}
sensitive then there exists $S\subset\mathbb{G}$ such that $h_{top}%
^{S}(T\mathcal{)>}0.$
\end{theorem}

\begin{proof}
Let $(X,T)$ be \underline{diam-mean} sensitive with sensitive constant
$\varepsilon.$ Let $\mathcal{U}:=\left\{  U_{1},...U_{N}\right\}  $ be a
finite open cover with balls with diameter smaller than $\varepsilon/2.$

We will define the sequence $S^{n}=\left\{  s_{1},...,s_{n}\right\}  $
inductively with $s_{1}=1$. For every $n\in\mathbb{N}$ we define
$L_{n}:=\left\{  \cap_{i=1}^{n}T^{-s_{i}}U_{v_{i}}\neq\emptyset:v\in\left\{
1,...,N\right\}  ^{S^{n}}\right\}  $ ($N$ is the size of the cover)$.$ We
denote by $L_{n}^{\prime}:=\left\{  A_{k}^{n}\right\}  _{k\leq N(L_{n})}$ a
subcover of $L_{n}$ of minimal cardinality$.$ We define the function
$f:L_{n}^{\prime}\rightarrow2^{\mathbb{G}}$ as follows; $m\in f(A_{k})$ if and
only if there exists $x,y\in A_{k}\diagdown\cup_{j<k}\overline{A}_{j}$ such
that $d(T^{m}x,T^{m}y)>\varepsilon.$

Assume $S^{n}$ is defined. Since $(X,T)$ is \underline{diam-mean} sensitive we
have that $\underline{D}(f(U))>\varepsilon$ for every $U\in L_{n}^{\prime}.$
By Lemma \ref{abc} there exists $g\in\mathbb{G}$ such that $\frac{\left\vert
\left\{  U\in L_{n}^{\prime}:g\in f(U)\right\}  \right\vert }{N(L_{n}%
)}>\varepsilon/2;$ we define $s_{n+1}:=g.$ The definition of $f$ implies that
$\frac{N(L_{n+1})}{N(L_{n})}>1+\varepsilon/2.$ Let $S^{\infty}:=\cup
_{n\in\mathbb{N}}S^{n}.$ Since $N(L_{n})=N(\vee_{i\in S^{n}}T^{-s_{i}%
}(\mathcal{U))},$ we conclude that $h_{top}^{S^{\infty}}(T,\mathcal{U)}>0.$
\end{proof}

\begin{corollary}
\label{nullweak}Let $(X,T)$ be a minimal TDS. If $(X,T)$ is null then it is
\underline{diam-mean} equicontinuous.
\end{corollary}

\begin{proof}
Apply Theorem \ref{seqentr} and Proposition \ref{lowerd}.
\end{proof}

Every minimal null TDS is almost automorphic (see \cite{huang2003null} and
\cite{huang2006tame}). Using the previous corollary and Corollary
\ref{regularity} we obtain a stronger result for subshifts.

\begin{corollary}
\label{nullquasi}Let $(X,\sigma)$ be a minimal subshift. If $(X,\sigma)$ is
null then it is a regular almost automorphic subshift and hence diam-mean equicontinuous.
\end{corollary}

\bigskip The converse of this result is not true (Example \ref{quasinotnull}).

In Corollary \ref{principal} we saw that an ergodic TDS is $\mu-$null if and
only if it is $\mu-$mean equicontinuous. If $(X,T)$ is mean equicontinuous and
$\mu$ is an ergodic measure then $(X,T)$ is $\mu-$mean equicontinuous, and
hence it has zero entropy. This implies that mean equicontinuous and diam-mean
equicontinuous systems have zero topological entropy.

Surprisingly it was shown in \cite{li2013mean} that transitive almost mean
equicontinuous subshifts can have positive entropy.

\bigskip

\subsubsection{\ Tightness and mean distality}

\begin{definition}
\label{tight}Let $(X,T)$ be a TDS and $\mu$ an invariant measure. We say
$(X,T)$ is \textbf{mean distal} if $d_{b}(x,y)>0$ for every $x\neq y\in X,$
and we say $(X,T)$ is $\mu-$\textbf{tight} if there exists $X^{\prime}$ such
that $\mu(X^{\prime})=1$ and $d_{b}(x,y)>0$ for every $x\neq y\in X^{\prime}.$
\end{definition}

Mean distal systems were studied by Ornstein-Weiss in \cite{MR2075122}. They
showed that any tight measure preserving $\mathbb{Z-}$TDS has zero entropy
(assuming the system has finite entropy)\textbf{. }A $\mathbb{Z}_{+}%
\mathbb{-}$TDS has zero topological entropy if and only if it is the factor of
a mean distal TDS (see \cite{downarowicz2012forward} and \cite{MR2075122}).

Let $(X,T)$ be a mean equicontinuous TDS. Since $f_{eq}=f_{b}$ we have that if
$(X,T)$ is mean distal then $f_{eq}$ is 1-1 hence $(X,T)$ is equicontinuous.
So mean equicontinuity and mean distality are both considered rigid
properties, and a TDS satisfies both properties if and only if it is equicontinuous.

\begin{proposition}
A mean equicontinuous TDS\ is mean distal if and only if it is equicontinuous.
\end{proposition}

A measure theoretic version of mean distality was also defined. Let $(X,T)$ be
a transitive almost automorphic diam-mean equicontinuous TDS and $\mu$ an
ergodic measure$.$ Using Theorem \ref{aa} it is not hard to see that $(X,T)$
is $\mu-$tight.

\bigskip

It is curious that $\mu-$tightness represents rigid motion and that $\mu-$mean
expansiveness ($\mu\times\mu\left\{  (x,y):d_{b}(x,y)>\varepsilon\right\}
=1$) represents very sensitive behaviour.

\subsection{Counter-examples}

The following example shows there are mean equicontinuous not diam-mean
equicontinuous TDS. We do not have a transitive counterexample.

\begin{example}
\label{weaknotquasi}Let $X\subset\left\{  0,1\right\}  ^{\mathbb{Z}_{+}}$ be
the subshift consisting of sequences that contain at most one $1.$ For every
$x,y\in X$, $d_{b}(x,y)=0,$ so $(X,\sigma)$ is mean equicontinuous.
Nonetheless, for every $\varepsilon>0,$ \underline{$D$}$\left\{
i\in\mathbb{Z}_{+}:\exists x\in B_{\varepsilon}(0^{\infty})\text{ s.t. }%
x_{i}=1\right\}  =1$ so $0^{\infty}$ is not a diam-mean equicontinuous point
(even not a \underline{diam}-equicontinuous point).
\end{example}

In \cite{kerrtopological} the relationship between independence and entropy
was studied. We will make use of their characterization of null systems.

\begin{definition}
Let $(X,T)$ be a TDS, and $A_{1},A_{2}\subset X$. \ We say $S\subset
\mathbb{G}$ is an \textbf{independence set for }$\left(  A_{1},A_{2}\right)  $
if for every non-empty finite subset $F\subset S$ we have%
\[
\cap_{i\in F}T^{-i}A_{v(i)}\neq\emptyset
\]
for any $v\in\left\{  1,2\right\}  ^{F}.$
\end{definition}

\begin{theorem}
[\cite{kerrtopological}]\label{kerr}Let $(X,T)$ be TDS. The system $(X,T)$ is
not null if and only if there exists $x,y\in X$ (with $x\neq y$)$,$ such that
for all neighbourhoods $U_{x}$ of $x$ and $U_{y}$ of $y$ there exists an
arbitrarily large finite independence set for $\left(  U_{x},U_{y}\right)  .$
\end{theorem}

The following example shows there are transitive non-null mean equicontinuous systems.

\begin{example}
\label{notnull}Let $S=\left\{  2^{n}\right\}  _{n=1}^{\infty}$,
\[
Y:=\left\{  x\in\left\{  0,1\right\}  ^{\mathbb{Z}_{+}}:x_{i}=0\text{ if
}i\notin S\right\}  ,
\]
and $X$ the shift-closure of $Y.$ For every $x,y\in X$ we have $d_{b}(x,y)=0,
$ hence $(X,\sigma)$ is mean equicontinuous. Nonetheless, since $S$ is an
infinite independence set for $\left(  \left\{  x_{0}=0\right\}  ,\left\{
x_{0}=1\right\}  \right)  $ we conclude that $(X,\sigma)$ is not null.
\end{example}

A $\mathbb{Z}_{+}-$subshift is \textbf{Toeplitz} if and only if it is the
orbit closure of a \textbf{regularly recurrent point}, i.e. $x\in X$ such that
for every $j>0$ there exists $m>0$ such that $x_{j}=x_{j+im}$ for all
$i\in\mathbb{Z}_{+}.$ Toeplitz subshifts are precisely the minimal subshifts
that are almost 1-1 extensions of odometers (for $\mathbb{Z}_{+}\mathbb{-}%
$actions see \cite{surveyodometers}, for finitely generated discrete group
actions see \cite{godometers}).

Given a Toeplitz subshift and a regularly recurrent point $x\in X,$ there
exists a set of pairwise disjoint arithmetic progressions $\left\{
S_{n}\right\}  _{n\in\mathbb{N}}$ (called the \textbf{periodic structure})
such that $\cup_{n\in\mathbb{N}}S_{n}=\mathbb{Z}_{+}$, $x_{i}$ is constant for
every $i\in$ $S_{n},$ and every $S_{n}$ is maximal in the sense that there is
no larger arithmetic progression where $x_{i}$ is constant$.$ Let $x $ be a
regularly recurrent point and $X$ the orbit closure. We have that $(X,\sigma)$
is regular if and only if $\sum\limits_{n\in\mathbb{N}}\underline{D}(S_{n})=1$
(see \cite{surveyodometers})$.$

The following example shows that the converse of Corollary \ref{nullquasi}
does not hold.

\begin{example}
\label{quasinotnull}There exists a regular Toeplitz subshift (hence diam-mean
equicontinuous) with positive sequence entropy.
\end{example}

\begin{proof}
For every $n\in\mathbb{N}$ let $w^{n}$ be a finite word that contains all
binary words of size $n.$ We denote the concatenation of $w^{n}$ by
$w^{n,\infty}\in\left\{  0,1\right\}  ^{\mathbb{Z}_{+}}.$

We define the sequence $\left\{  j_{n}\right\}  \subset\mathbb{N}$ inductively
with $j_{1}=0$ and $j_{n+1}:=\min\left\{  \cup_{m\leq n}\cup_{k\in\mathbb{N}%
}\left\{  k2^{m}+j_{m}\right\}  \right\}  ^{c}.$

Let $x\in\left\{  0,1\right\}  ^{\mathbb{Z}_{+}}$ be the point such that for
every $n\in\mathbb{N}$ we have that $x_{j_{n}+i2^{n}}=w_{i}^{n,\infty}$ for
all $i\in\mathbb{Z}_{+}.$

We define $X$ as the orbit closure of $x.$ Since $x$ is regularly recurrent we
obtain that $X$ is a Toeplitz subshift (hence almost automorphic). By using
the condition for regularity using the periodic structure (see comment before
this proposition) we obtain that $(X,\sigma)$ is regular. Hence by Theorem
\ref{aa} we get that $(X,\sigma)$ is diam-mean equicontinuous.

On the other hand $w^{k}$ contains all the binary words of size $k$. This
implies there exists arbitrarily long independence sets for $\left(  \left\{
x_{0}=0\right\}  ,\left\{  x_{0}=1\right\}  \right)  .$ Using Theorem
\ref{kerr} we conclude $(X,\sigma)$ is not null.
\end{proof}

Another class of rigid TDS are the tame systems introduced in
\cite{kohler1995enveloping} and \cite{Glasner_ontame}. These systems were
characterized in \cite{kerrtopological} similarly to Theorem \ref{kerr} but
with infinitely large independence sets. This means that Example \ref{notnull}
is also not tame (note it is not minimal). The example in Section 11 in
\cite{kerrtopological} is a tame non-regular Toeplitz subshift; this means
there are tame minimal systems that are not diam-mean equicontinuous.
Nonetheless we do not know if every minimal mean equicontinuous system is tame.

\section{\bigskip Amenable semigroup actions}

All of our results can be stated for more general group actions. In this
section we state the generality of the results. All the results need
amenability and all the results hold for countable abelian discrete actions.

Amenable groups are usually defined with invariant means (see
\cite{pier1984amenable}). We give an equivalent definition that is more useful
for our paper.

\begin{definition}
\bigskip Let $\left(  \mathbb{G}\text{,}+\right)  $ be a locally compact
semigroup that has an invariant measure $\nu.$ We say $\mathbb{G}$ is
\textbf{amenable} if there exists a\textbf{\ F\o lner sequence, }i.e. a
sequence of measurable sets with finite measure $\left\{  F_{n}\right\}
\subset\mathbb{G}$, such that for any $i\in\mathbb{G}$ we have that%
\[
\lim_{n\rightarrow\infty}\frac{\nu(\left(  i+F_{n}\right)  \bigtriangleup
F_{n})}{\nu(F_{n})}=0.
\]

\end{definition}

If the group is countable then any invariant measure is a counting measure. In
this case if $F$ is compact then $F$ is finite and $\nu(F)=\left\vert
F\right\vert .$

A semigroup $\left(  \mathbb{G}\text{,}+\right)  $ is \textbf{left
cancellative} if whenever $a+b=a+c$ we have that $b=c.$ Every group is a left
cancellative semigroup.

Throughout this section we assume $\left(  \mathbb{G}\text{,}+\right)  $ is an
amenable left cancellative locally compact semigroup with identity ($0$) that
has an invariant measure $\nu$ (a group is always left cancellative and there
exists (left) invariant measures known as Haar measures).

From now on $\mathbb{G}$ implicitly represents a pair consisting of a
semigroup and a F\o lner sequence.

Abelian semigroups are amenable. In particular $\mathbb{Z}_{+}^{d},$ and
$\mathbb{R}_{+}^{d}$ are amenable (and left cancellative); in these cases we
associate the cubes $\left[  0,n\right]  ^{d}$ as the F\o lner sequence. For
$\mathbb{Z}^{d},$ and $\mathbb{R}^{d}$ we associate $\left[  -n,n\right]
^{d}.$

\begin{definition}
Let $S\subset\mathbb{G}$. We define \textbf{lower density} as%
\[
\underline{D}(S):=\liminf_{n\rightarrow\infty}\frac{\nu(S\cap F_{n})}%
{\nu(F_{n})},
\]
and \textbf{upper density }as
\[
\overline{D}(S):=\limsup_{n\rightarrow\infty}\frac{\nu(S\cap F_{n})}{\nu
(F_{n})}.
\]

\end{definition}

\begin{lemma}
Let $S,S^{\prime}\subset\mathbb{G}$, $i\in\mathbb{G}$, and $F\subset
\mathbb{G}$ a finite set$.$ We have that

$\cdot\underline{D}(S)=\underline{D}(i+S)$ and $\overline{D}(S)=\overline
{D}(i+S).$

$\cdot$\underline{$D$}$(S)+\overline{D}(S^{c})=1.$

$\cdot$If $\underline{D}(S)+\overline{D}(S^{\prime})>1$ then $S\cap S^{\prime
}\neq\emptyset.$

$\cdot$If $\mathbb{G}$ is countable%
\begin{align*}
\underline{D}(S)  &  :=\liminf_{n\rightarrow\infty}\frac{\left\vert S\cap
F_{n}\diagdown F\right\vert }{\left\vert F_{n}\diagdown F\right\vert },\text{
and}\\
\overline{D}(S)  &  :=\limsup_{n\rightarrow\infty}\frac{\left\vert S\cap
F_{n}\diagdown F\right\vert }{\left\vert F_{n}\diagdown F\right\vert }.
\end{align*}

\end{lemma}

\begin{proof}
To prove the first property we have that
\begin{align*}
\underline{D}(S)  &  =\liminf_{n\rightarrow\infty}\frac{\nu(S\cap F_{n})}%
{\nu(F_{n})}\\
&  =\liminf_{n\rightarrow\infty}\frac{\nu(-i+i+S\cap F_{n})}{\nu(F_{n})}\\
&  =\liminf_{n\rightarrow\infty}\frac{\nu(i+S\cap i+F_{n})}{\nu(F_{n})}\\
&  =\liminf_{n\rightarrow\infty}\frac{\nu(i+S\cap i+F_{n})}{\nu(i+F_{n})}%
\frac{\nu(i+F_{n})}{\nu(F_{n})}\\
&  =\underline{D}(i+S).
\end{align*}
The other properties are also easy to show.
\end{proof}

A $\mathbb{G}$\textbf{\ -measure preserving transformation ($\mathbb{G-}$MPT)}
is a triplet $(M,\mu,T)$ where $(M,\mu)$ is a probability space and
$T:=\left\{  T^{i}:i\in\mathbb{G}\right\}  $ is a $\mathbb{G}-$ measure
preserving action on $M.$ When we say a system is ergodic it means it is
measure preserving and ergodic.

For some of the results in this section we will assume the pointwise ergodic
theorem holds.

Let $(M,\mu,T)$ be an ergodic $\mathbb{G}-$MPT. Let $A$ be a $\mu-$measurable
set. We say $x\in M$ is a \textbf{generic point for }$\mathbf{A}$%
\textbf{\ }if
\[
\lim_{n\rightarrow\infty}\frac{\nu\left\{  i\in F_{n}:T^{i}x\in A\right\}
}{\nu(F_{n})}=\mu(A).
\]

\bigskip We say a $\mathbb{G}-$MPT satisfies the \textbf{pointwise ergodic
theorem} if for every measurable set $A$ almost every point is a generic point
for $A$. The pointwise ergodic theorem was originally proved for
$\mathbb{Z}_{+}\mathbb{-}$systems by Birkhoff. It also holds if $\mathbb{G=Z}%
_{+}^{d}$ \cite{gorodnik2009ergodic}$.$ Every second countable amenable group
has a F\o lner sequence that satisfies the pointwise ergodic theorem
\cite{lindenstrauss2001pointwise} (note that this is not satisfied for every
F\o lner sequence). For other conditions when this holds see Chapter 6.4 in
\cite{krengel1985ergodic}.

\bigskip

\subsection{Results that hold for any $\mathbb{G}$}

The following results hold for any $\mathbb{G}$ (that is amenable left
cancellative locally compact with identity ($0$)): Theorem \ref{meaneq},
Theorem \ref{strongsensitive}, Proposition \ref{isomorphism}, Theorem
\ref{sensipairs}, Theorem \ref{discrete}, Theorem \ref{dichotomy}, and Lemma
\ref{abc}.

\bigskip

The proofs of Theorem \ref{meaneq}, Theorem \ref{strongsensitive}, Proposition
\ref{isomorphism}, Theorem \ref{sensipairs}, Theorem \ref{dichotomy} are identical.

\bigskip Lemma \ref{abc} also holds in general but the proof is slightly different.

\begin{lemma}
Let $K$ be a finite set, $\varepsilon>0,$ and $h:K\rightarrow2^{\mathbb{G}} $
such that $\underline{D}(h(k))>\varepsilon$ for every $k\in F.$ There exist
$K^{\prime}\subset K$ and $i\in\mathbb{G}$ with $\left\vert K^{\prime
}\right\vert \geq\varepsilon\left\vert K\right\vert /2$ such that $i\in h(k)$
for every $k\in K^{\prime}.$
\end{lemma}

\begin{proof}
There exists $n_{0}\in\mathbb{N}$ such that%
\[
\frac{\nu(h(k)\cap F_{n_{0}})}{\nu(F_{n_{0}})}\geq\varepsilon/2\text{
}\ \ \text{for every }k\in K.
\]

Let $\mathcal{B}$ be a finite family of disjoint subsets of $F_{n_{0}},$ such
that $\nu(B)$ is constant for every $B\in\mathcal{B}$ and for every $k\in K$
there exists $\mathcal{B}_{k}\subset\mathcal{B}$ such that
\[
\nu(h(k)\cap F_{n_{0}})=\nu(\cup_{B\in\mathcal{B}_{k}}B).
\]

This means that
\[
\sum_{B\in\mathcal{B}}\left\vert k\in K:B\subset h(k)\right\vert =\sum
_{k\in\mathcal{K}}\left\vert B\in\mathcal{B}:B\subset h(k)\right\vert
\geq\frac{\varepsilon}{2}\left\vert K\right\vert \left\vert \mathcal{B}%
\right\vert .
\]

\end{proof}

Using this we obtain Theorem \ref{seqentr}, and Corollary \ref{nullquasi} for
any $\mathbb{G}$ such that $(G_{eq},\nu_{eq},T_{eq})$ satisfies the pointwise
ergodic theorem.

A $\mathbb{G-}$measure preserving transformation on a probability space
$(M,\mu,T)$ generates a family of unitary linear operator on the Hilbert space
$L^{2}(M,\mu),$ by $U_{T^{i}}:f\mapsto f\circ T^{i}.$ We say $(M,\mu,T)$ has
\textbf{discrete spectrum} if $L^{2}(M,\mu)$ is the direct sum of finite
dimensional $U_{T}$-invariant subspaces. Mackey proved Halmos and Von Neumann
holds for locally compact group actions \cite{mackey1964ergodic}. This implies
Theorem \ref{discrete} holds for any $\mathbb{G}$.

\subsection{Results that hold for countable discrete abelian $\mathbb{G}$}

The following results hold for countable discrete abelian $\mathbb{G}$:
Corollary \ref{principal}, Theorem \ref{aa}, and Corollary \ref{nullquasi}.
The proof of Proposition \ref{entropy} uses a result from
\cite{KerrMeasurable} that holds for discrete actions, and the pointwise
ergodic theorem. Also it uses the fact that the action is abelian. This
implies this result holds for countable discrete abelian $\mathbb{G}$ that
satisfy the pointwise ergodic theorem.

For $\mathbb{Z}_{+}$-actions, Kushnirenko proved that an ergodic system is
$\mu-$null if and only if it has discrete spectrum \cite{0036-0279-22-5-R02}.
This result was generalized for discrete actions by Kerr-Li in
\cite{KerrMeasurable}. With this we obtain that Corollary \ref{principal}
holds for discrete countable abelian actions that satisfy the pointwise
ergodic theorem.

Generalized shift systems are defined for countable groups. Theorem \ref{aa}
holds whenever $\mathbb{G}$ is countable.

If $\mathbb{G}$ is an abelian group then every minimal null $\mathbb{G}-$TDS
is almost automorphic (see \cite{huang2003null} and \cite{huang2006tame}).
This implies Corollary \ref{nullquasi} holds for countable abelian groups.

\bibliographystyle{aabbrv}
\bibliography{acompat,camel}

\end{document}